\documentclass[a4paper, DIV=12, 11pt]{amsart}
\usepackage[latin1]{inputenc}
\usepackage{amssymb, amsthm, amsmath, thmtools, mathtools}
\usepackage{amsfonts}
\usepackage[abbrev]{amsrefs} 

\usepackage{graphicx}
\usepackage{thmtools}
\usepackage{hyperref}
\usepackage[bottom, marginal]{footmisc}
\usepackage{todonotes}

\usepackage{lipsum}

\usepackage{tikz}
\usetikzlibrary{matrix}

\declaretheoremstyle[headfont=\normalfont]{normalhead}
\newtheorem{lemma}{Lemma}[section]
\newtheorem{theorem}[lemma]{Theorem}
\newtheorem{proposition}[lemma]{Proposition}
\newtheorem{corollary}[lemma]{Corollary}
\newtheorem{definition}[lemma]{Definition}
\newtheorem{remark}[lemma]{Remark}
\newtheorem{example}[lemma]{Example}

\newtheorem*{acknowledgement}{Acknowledgement}

\newcounter{mt}

\newtheorem{maintheorem}[mt]{Theorem}

\newcommand{\R}{\mathbb{R}}
\newcommand{\C}{\mathbb{C}}

\DeclareMathOperator{\VConv}{VConv}
\DeclareMathOperator{\Conv}{Conv}
\DeclareMathOperator{\vol}{vol}

\DeclareMathOperator{\supp}{supp}

\DeclareMathOperator{\diam}{diam}

\DeclareMathOperator{\Hess}{Hess}
\DeclareMathOperator{\GL}{GL}

\DeclareMathOperator{\GW}{\mathrm{GW}}

\DeclareMathOperator{\Gr}{\mathrm{Gr}}

\DeclareMathOperator{\Sym}{\mathrm{Sym}}
\DeclareMathOperator{\Poly}{\mathcal{P}}

\DeclareMathOperator{\SO}{\mathrm{SO}}

\DeclareMathOperator{\MAVal}{\mathrm{MAVal}}

\DeclareMathOperator{\MA}{\mathrm{MA}}

\DeclareMathOperator{\LV}{\mathrm{LV}}

\renewcommand{\Im}{\operatorname{Im}}

\renewcommand{\P}{\mathrm{P}}
\newcommand{\M}{\mathcal{M}}

\newcommand{\A}{\mathrm{A}}
\newcommand{\Aff}{\mathrm{Aff}}

\author{Jonas Knoerr}
\title{Polynomial local functionals on convex functions}
\date{}

\newcommand{\Addresses}{{
		\bigskip
		\footnotesize
		
		Jonas Knoerr, \textsc{Institute of Discrete Mathematics and Geometry, TU Wien, Wiedner Hauptstrasse 8-10, 1040 Wien, Austria}\par\nopagebreak
		\textit{E-mail address}: \texttt{jonas.knoerr@tuwien.ac.at}
		
		\medskip
	}}
	
\makeatletter
\def\blfootnote{\xdef\@thefnmark{}\@footnotetext}
\makeatother

\makeindex
\begin{document}
\maketitle
\begin{abstract}
	We show that every continuous local functional on the space of finite convex functions on $\mathbb{R}^n$ is a valuation. This relation is used to establish a homogeneous decomposition for the class of polynomial local functionals as well as a classification of translation or rigid motion invariant polynomial local functionals. In addition we discuss implications for the compact-open topology on the space of polynomial local functionals.
\end{abstract}


\sloppy

\section{Introduction}
	Given a locally compact Hausdorff topological space $Y$ and a set of functions $\mathcal{F}$ on $Y$, a map $\Psi:\mathcal{F}\rightarrow \mathcal{M}(Y)$ into the space of (complex) Radon measures on $Y$ is called \emph{locally determined} or a \emph{local functional} if for every $f,h\in\mathcal{F}$ such that $f|_U=h|_U$ for some open subset $U\subset Y$ the associated measures satisfy
	\begin{align*}
		\Psi(f)|_U=\Psi(h)|_U.
	\end{align*}
	Examples of local functionals naturally occur in various areas of mathematics, including the calculus of variations, differential geometry, and convex geometry. In particular, local functionals on Sobolev spaces have been investigated extensively over the last 50 years, compare \cite{AlbertiIntegralrepresentationlocal1993,BottaroOppezziMultipleintegralrepresentations1985,ButtazzoDalMasoIntegralrepresentationrelaxation1985,DalMasoEtAlIntegralrepresentationclass1994,EssebeiEtAlIntegralrepresentationlocal2023,MaioneEtAlconvergencefunctionals2020}.\\
	
	In this article, we consider continuous local functionals on the space $\Conv(\R^n,\R)$ of all finite convex functions $f:\R^n\rightarrow\R$. This space is naturally equipped with a metrizable topology induced by locally uniform convergence, which in this setting coincides with epi-convergence or pointwise convergence (compare the references in \autoref{subsection:convexFunctions}). We identify $\M(\R^n)\cong C_c(\R^n)'$ with the topological dual space of the space of compactly supported continuous functions on $\R^n$ equipped with its usual inductive topology, and we equip $\M(\R^n)$ with the weak* topology.\\
	As an example, consider the functionals $\mathcal{E}^{a,b}_c$, $a,b,c\in\mathbb{N}$, $0\le c\le n$, given for $f\in C^2(\R^n)$ and bounded Borel sets $B\subset \R^n$ by
	\begin{align}
		\label{eq:definitionEabc}
		\mathcal{E}^{a,b}_{c}(f;B)=\int_{B}f(x)^a|\nabla f(x)|^{2b} \det(D^2f(x)[c],Id[n-c])dx,
	\end{align}
	where $D^2f$ denotes the Hessian of $f$, $\det$ is the mixed discriminant, and the matrices $D^2f(x)$ and $Id$ are taken with multiplicity $c$ and $n-c$ respectively. Then $\mathcal{E}^{a,b}_c$ extends to a continuous local functional on $\Conv(\R^n,\R)$ (compare \autoref{maintheorem:rigidMotionInvariant} and \autoref{section:ConstructionLocalFunc} below), and this family includes various well known functionals from the calculus of variations, for example the $k$-Hessian energy $\mathcal{E}^{1,0}_k$, the Dirichlet energy $\mathcal{E}^{0,1}_0$, the Hessian measures $\mathrm{Hess}_k=\mathcal{E}^{0,0}_k$, and the real Monge--Amp\`ere operator $\MA=\mathcal{E}^{0,0}_n$. Measure-valued functionals of this type have recently become an important tool in the theory of valuations on convex functions \cite{AleskerValuationsconvexfunctions2019,ColesantiEtAlHessianvaluations2020,ColesantiEtAlhomogeneousdecompositiontheorem2020}, in particular in the construction of invariant valuations \cite{ColesantiEtAlHadwigertheoremconvex2024,ColesantiEtAlHadwigertheoremconvex2022,MouamineMussnigVectorialHadwigerTheorem2025,KnoerrSingularValuationsHadwiger2025}. Here, a functional $\mu:\Conv(\R^n,\R)\rightarrow \mathcal{A}$ into an Abelian semi-group $\mathcal{A}$ is called a valuation if 
	\begin{align*}
		\mu(f)+\mu(h)=\mu(f\vee h)+\mu(f\wedge h)
	\end{align*}
	for all $f,h\in \Conv(\R^n,\R)$ such that the pointwise minimum $f\wedge h$ belongs to $\Conv(\R^n,\R)$ (note that the pointwise maximum $f\vee h$ is always convex). These constructions rely on the fact that a variety of Monge--Amp\`ere-type operators can be considered as valuations on $\Conv(\R^n,\R)$ with values in $\M(\R^n)$. This non-obvious fact was first pointed out by Alesker in \cite{AleskerValuationsconvexsets2005}, who used results by B{\l}ocki \cite{BlockiEquilibriummeasureproduct2000} for the complex Monge--Amp\`ere operator to construct valuations on convex bodies.\\
	The starting point of our investigation of local functionals on convex functions is the following general result.
	\begin{maintheorem}
		\label{maintheorem:LocalFunctionalValuation}
		Let $\Psi:\Conv(\R^n,\R)\rightarrow\M(\R^n)$ be a continuous local functional. Then $\Psi$ is a valuation.
	\end{maintheorem}
	Let us remark that a similar phenomenon was observed by Weil in \cite{WeilIntegralgeometrytranslation2015,WeilIntegralgeometrytranslation2017} for a class of continuous and translation invariant functionals on convex bodies satisfying related locality properties.\\
	This connection between local functionals and valuations does not exist for arbitrary spaces of functions. For example, Wang showed in \cite{WangSemivaluations$BVRn$2014} that the variational measure is a local functional on functions of bounded variation that does not satisfy the valuation property. A similar phenomenon can occur for non-continuous local functionals on convex functions, compare \autoref{example:nonContLocaFunc} below. In general, the relation between notions of locality and the valuation property is thus more subtle.\\
	
	\autoref{maintheorem:LocalFunctionalValuation} allows us to treat local functionals as locally determined measure-valued valuations and, consequently, to apply tools from valuation theory to local functionals. Let $\A(n,\R)$ denote the space of all affine functions $\ell:\R^n\rightarrow\R$.
	\begin{definition}
		We call a local functional $\Psi:\Conv(\R^n,\R)\rightarrow\M(\R^n)$ polynomial of degree at most $d\in\mathbb{N}$ if the map
		\begin{align*}
			\A(n,\R)&\rightarrow \M(\R^n)\\
			\ell&\mapsto \Psi(f+\ell)
		\end{align*}
		is a polynomial of degree at most $d$ for every $f\in\Conv(\R^n,\R)$. 
	\end{definition}
	More generally we will call a local functional on $\Conv(\R^n,\R)$ a polynomial local functional if it is polynomial of degree $d$ for some $d\in\mathbb{N}$. Note that the examples in Eq.~\eqref{eq:definitionEabc} all belong to this class. \\
	Within valuation theory, this property naturally occurs as a generalization of translation invariance (which corresponds to the case $d=0$) or equivariance, compare \cite{Aleskermultiplicativestructurecontinuous2004, PukhlikovKhovanskiiFinitelyadditivemeasures1992}. Local functionals on Sobelev spaces that are invariant with respect to the addition of higher order polynomials have also been considered in \cite{AlbertiIntegralrepresentationlocal1993,ButtazzoDalMasocharacterizationnonlinearfunctionals1985,ButtazzoDalMasoIntegralrepresentationrelaxation1985}.\\
		
	This article is the first in a two part series (see \cite{KnoerrIntegralrepresentationpolynomial2026} for the second part) with the goal to establish general properties of the space of polynomial local functionals as well as integral representations. In this first part, we focus on general structural results for spaces of polynomial local functionals, whereas the second part will focus on integral representations of local functionals and density results.\\
	
	Let us discuss the results of this article. We denote by $\P_d\LV(\R^n)$ the space of all continuous polynomial local functionals on $\Conv(\R^n,\R)$ of degree at most $d$.	Let $\P_d\LV_k(\R^n)$ denote the subspace of all $k$-homogeneous local functionals, i.e. all $\Psi\in \P_d\LV(\R^n)$ that satisfy
	\begin{align*}
		\Psi(tf)=t^k\Psi(f)
	\end{align*}
	for all $f\in\Conv(\R^n,\R)$ and $t\ge 0$. From a general result on polynomial valuations based on work by Pukhlikov and Khovanski\u i \cite{PukhlikovKhovanskiiFinitelyadditivemeasures1992}, we obtain the following homogeneous decomposition.
	\begin{maintheorem}
		\label{maintheorem:HomDecompPdLV}
		\begin{align*}
			\P_d\LV(\R^n)=\bigoplus_{k=0}^{n+d} \P_d\LV_k(\R^n)
		\end{align*}
	\end{maintheorem}	
	Note that every element in $\P_d\LV_0(\R^n)$ is constant, so $\P_d\LV_0(\R^n)$ is isomorphic to $\M(\R^n)$. For the homogeneous component of degree $n+d$, we have the following characterization (the case $d=0$ was also obtained in \cite[Theorem 4.8]{KnoerrMongeAmpereoperators2024}). Let us identify $ \R\times(\R^n)^*\cong\A(n,\R)$ using the map $(s,v)\mapsto s+\langle v,\cdot\rangle$, and $\Sym^d(\A(n,\R)^*)\otimes\C$ with the space of $d$-homogeneous complex-valued polynomials on $\A(n,\R)$.
	\begin{maintheorem}
		\label{maintheorem:topDegree}
		For every $\Psi\in \P_d\LV_{n+d}(\R^n)$ there exists a unique function $\psi\in C(\R^n,\Sym^d(\A(n,\R)^*)\otimes\C)$ such that
		\begin{align*}
			\Psi(f;B)=\int_{B}\psi(x)[f(x),df(x)] \det(D^2f(x))dx
		\end{align*}
		for $f\in \Conv(\R^n,\R)\cap C^2(\R^n)$, $B\subset\R^n$ bounded Borel set. Conversely, the right hand side of this equation extends by continuity to an element of $\P_d\LV_{n+d}(\R^n)$ for every such function $\psi$.
	\end{maintheorem}
	We will establish integral representations for elements belonging to the intermediate degrees in \cite{KnoerrIntegralrepresentationpolynomial2026}. This description relies on the following characterization of the subspace of translation invariant polynomial local functionals.\\
	Let $\pi$ denote the representation of the group $\Aff(n,\R)$ of all invertible affine transformations of $\R^n$ on $\P_d\LV(\R^n)$ given for $g\in \Aff(n,\R)$ and $\Psi\in \P_d\LV(\R^n)$ by
	\begin{align}
		\label{eq:groupAction}
		[\pi(g)\Psi](f;B)=\Psi(f\circ g;g^{-1}(B))
	\end{align}
	for $f\in\Conv(\R^n,\R)$, $B\subset\R^n$ bounded Borel set. We denote by $\P_d\LV(\R^n)^{tr}$ the subspace of all polynomial local functionals of degree at most $d$ that are translation invariant with respect to this action. The following result identifies this space with a certain finite dimensional space of polynomials. In order to state it, let $\Poly(V)$ denote the space of $\C$-valued polynomials on a (real) finite dimensional vector space $V$ and $\Poly_d(V)\subset\Poly(V)$ the subspace of polynomials of degree at most $d$. Specifically, we will consider the space $\mathrm{M}_n\subset \Poly(\Sym^2(\R^n))$ spanned by all $k$-minors, $0\le k\le n$, of symmetric $(n\times n)$-matrices, where the $0$-minor of a matrix is $1$ by definition. 
	\begin{maintheorem}
		\label{maintheorem:translationInvariantCase}
		For every $\Psi\in\P_d\LV(\R^n)^{tr}$ there exists a unique polynomial $P_\Psi\in \Poly_d(\A(n,\R))\otimes \mathrm{M}_n$ such that
		\begin{align*}
			\Psi(f;B)=\int_{B}P_\Psi(f(x),df(x),D^2f(x))dx
		\end{align*}
		for all $f\in \Conv(\R^n,\R)\cap C^2(\R^n)$, $B\subset\R^n$ bounded Borel set. Conversely, the right hand side of this equation extends by continuity to a unique element in $\P_d\LV(\R^n)^{tr}$ for every polynomial in $\Poly_d(\A(n,\R))\otimes \mathrm{M}_n$.
	\end{maintheorem}
	As an application of this result, we obtain a classification of all rigid motion invariant elements in $\P_d\LV(\R^n)$. In addition to the elements defined by Eq.~\eqref{eq:definitionEabc}, consider the functionals $\mathcal{F}^{a,b}_c$, $a\ge 0, b\ge 1$, $1\le c\le n-1$, given for $f\in\Conv(\R^n,\R)\cap C^2(\R^n)$, $B\subset \R^n$ bounded Borel set, by	
	\begin{align*}
		&\mathcal{F}^{a,b}_c(f;B)\\
		&\quad=\int_{B}f(x)^a|\nabla f(x)|^{2(b-1)} \det(D^2f(x)[c],\nabla f(x)\cdot \nabla f(x)^T,Id[n-c-1])dx.
	\end{align*}
	\autoref{maintheorem:translationInvariantCase} implies that $\mathcal{F}^{a,b}_c$ extends to a continuous local functional on $\Conv(\R^n,\R)$.
	\begin{maintheorem}
		\label{maintheorem:rigidMotionInvariant}
		Assume that $n\ge 3$ and let $\Psi\in \P_d\LV_k(\R^n)$. Then $\Psi$ is rigid motion invariant if and only if there exist unique constants $c_{ij},d_{ij}\in\C$ such that
		\begin{align*}
			\Psi=&\sum_{i=\max(k-d,0)}^{\min(n,k)}\sum_{j=0}^{\lfloor\frac{\min(k-i,d)}{2}\rfloor}c_{ij} \mathcal{E}^{k-2j-i,j}_{i}+\sum_{i=\max(k-d,1)}^{\min(n-1,k)}\sum_{j=1}^{\lfloor\frac{\min(k-i,d)}{2}\rfloor}d_{ij} \mathcal{F}^{k-2j-i,j}_{i}.
		\end{align*}
	\end{maintheorem}
	For $n=2$ a similar result holds, however, the classification involves an additional family of functionals, compare the discussion in \autoref{section:RigidMotionInvariant}.\\ 
	
	The last type of results of this article concerns the topological properties of the space $\P_d\LV(\R^n)$. This space may be equipped with a variety of topologies by taking the compact-open topology with respect to some topology on $\M(\R^n)$. We discuss three different topologies for general spaces of measure-valued functionals in \autoref{section:topologyFunctionals}, but will focus on the following consequence of the relation to valuations on convex functions for the topology induced by local convergence with respect to the local total variation on $\M(\R^n)$. Consider the semi-norms on $\P_d\LV(\R^n)$ given for $A\subset\R^n$ compact and $K\subset \Conv(\R^n,\R)$ compact by
	\begin{align}
		\label{eq:DefSemiNormBoundedTop}
		\|\Psi\|_{A;K}:=\sup\limits_{\substack{f\in K,\phi\in C_c(\R^n)\\ \supp\phi\subset A, \|\phi\|_\infty\le 1}}|\Psi(f;\phi)|.
	\end{align}
	Here, we denote by $\phi\mapsto \Psi(f;\phi)$ the integration functional on $C_c(\R^n)$ induced by the measure $\Psi(f)$. It follows from the principle of uniform boundedness that these semi-norms are well defined, compare \autoref{lemma:weakContImpliesBounded} below. We will call the topology induced by these semi-norms the \emph{compact-to-bounded topology} and refer to \autoref{section:topologyFunctionals} for an explanation of the terminology. 
	It is not difficult to see that $\P_d\LV(\R^n)$ is complete with respect to these semi-norms. Our final result shows that this topology is actually metrizable.
	\begin{maintheorem}\label{maintheorem:Frechet}
		The space $\P_d\LV(\R^n)$ is a Fr\'echet space with respect to the compact-to-bounded topology.
	\end{maintheorem}
	We give a construction of a countable family of semi-norms inducing the topology in \autoref{section:supportTopology}.

\subsection{Plan of the article}
	\autoref{section:preliminaries} reviews the necessary background for convex functions and polynomial valuations on convex functions. In \autoref{section:topologyFunctionals}, we investigate  different topologies on spaces of measure-valued functionals in view of their completeness and boundedness properties. This section is more general than required for this article and discusses in particular questions on the continuity of representations induced by the action of Lie groups on the underlying spaces, which are relevant to \cite{KnoerrIntegralrepresentationpolynomial2026,KnoerrTranslationinvariantarea2026}. The different topologies also give rise to several notions of smoothness for local functionals, which turn out to coincide.\\
	
	\autoref{section:localFunctionals} establishes general properties of local functionals and contains in particular the proofs of \autoref{maintheorem:LocalFunctionalValuation} and \autoref{maintheorem:HomDecompPdLV}. The relation to valuations on convex functions is then used to investigate the support properties of these functionals and the compact-to-bounded topology on $\P_d\LV(\R^n)$, including the proof of \autoref{maintheorem:Frechet}. In addition, we review a construction of certain measure-valued valuations based on the integration of differential forms over the differential cycle, which in particular allows us to construct the local functionals needed for the characterizations results.\\
	\autoref{section:classificationResults} contains the proof of \autoref{maintheorem:translationInvariantCase}, from which we obtain \autoref{maintheorem:rigidMotionInvariant} using some basic tools from the representation theory of $\SO(n)$ in combination with a classification of translation invariant polynomial local functionals of degree $0$ in terms of integration with respect to the differential cycle from \cite{KnoerrMongeAmpereoperators2024}.
	
	\begin{acknowledgement}
		I want to thank Monika Ludwig for her questions on the relation between the valuation property and locally determined functionals, which were the starting point for \autoref{maintheorem:LocalFunctionalValuation}.\\
		
		This research was funded in whole or in part by the Austrian Science Fund (FWF), \href{https://www.doi.org/10.55776/PAT4205224}{10.55776/PAT4205224}. 
	\end{acknowledgement} 
	
\section{Preliminaries}
\label{section:preliminaries}
Unless stated otherwise, all vector spaces are assumed to be complex vector spaces. In particular, for a finite dimensional real vector space $V$, we denote by $\Lambda^kV^*$ the space of complex-valued $k$-forms on $V$, and for a finite dimensional real manifold $M$, we denote by $\Omega^k(M)$ the space of complex-valued differential $k$-forms. If $V$ is a real vector space, we denote by $V_\C=V\otimes \C$ its complexification. If $F$ is a topological vector space, we let $F'$ denote its topological dual, i.e. the space of all continuous linear maps $F\rightarrow\C$.
\subsection{Convex functions}
\label{subsection:convexFunctions}
We refer to the monographs \cite{RockafellarConvexanalysis1997,RockafellarWetsVariationalanalysis1998} for a comprehensive background on convex functions and only collect the results required for the constructions in this article.\\	

The following is a special case of \cite{RockafellarWetsVariationalanalysis1998}*{9.14}.
\begin{proposition}
	\label{proposition_convex_functions_local_lipschitz_constants}
	Let $U\subset \R^n$ be a convex open subset and $f:U\rightarrow\R$ a convex function. If $X\subset U$ is a set with $X+\epsilon B_1\subset U$ and $f$ is bounded on $X+ B_\epsilon$, then $f$ is Lipschitz continuous on $X$ with Lipschitz constant $\frac{2}{\epsilon}\sup_{x\in X+B_\epsilon}|f(x)|$.
\end{proposition}
This result implies in particular that every function in $\Conv(\R^n,\R)$ is continuous. Recall that we equip the space $\Conv(\R^n,\R)$ with the topology of uniform convergence on compact subsets, which coincides with the topology induced by pointwise convergence or epi-convergence, compare \cite[Theorem 7.17]{RockafellarConvexanalysis1997}. This implies that $\Conv(\R^n,\R)$ is metrizable. We will require the following characterization of compactness in $\Conv(\R^n,\R)$.

\begin{proposition}[\cite{Knoerrsupportduallyepi2021}*{Proposition 2.4}]
	\label{proposition_compactness_Conv}
	A subset $U\subset \Conv(\R^n,\R)$ is relatively compact if and only if it is bounded on compact subsets of $\R^n$, i.e. if for any compact subset $A\subset \R^n$ there exists a constant $C(A)>0$ such that
	\begin{align*}
		\sup\limits_{x\in A}|f(x)|\le C(A)\quad \forall f\in U.
	\end{align*}
\end{proposition}

We denote by $\Aff(n,\R)$ the set of all invertible affine transformations of $\R^n$. Since every element in $\Conv(\R^n,\R)$ is continuous, the following is a direct consequence of the description of the topology on $\Conv(\R^n,\R)$ in terms of pointwise convergence.
\begin{lemma}\label{lemma:continuityActionAffonConv}
	The map
	\begin{align*}
		\Aff(n,\R)\times\Conv(\R^n,\R)&\rightarrow\Conv(\R^n,\R)\\
		(g,f)&\mapsto f\circ g^{-1}
	\end{align*}
	is continuous.
\end{lemma}

\subsection{Polynomial valuations on convex functions}
\label{section:polyVal}
Recall that we denote by $\A(n,\R)$ the space of all affine maps $\ell:\R^n\rightarrow\R$. By definition, a valuation $\mu:\Conv(\R^n,\R)\rightarrow F$ into a Hausdorff topological vector space $F$ is called a \emph{polynomial valuation} if there exists $d\in\mathbb{N}$ such that the map
\begin{align*}
	\A(n,\R)&\rightarrow F\\
	\ell&\mapsto \mu(f+\ell)
\end{align*}	
is a polynomial of degree at most $d$ on $\A(n,\R)$ with values in $F$ for every $f\in\Conv(\R^n,\R)$. In this case we will call $\mu$ a polynomial valuation of degree at most $d$. Using the inverse of the Vandermonde matrix, it is easy to see that this is equivalent to the existence of (unique) maps $Y_j:\Conv(\R^n,\R)\rightarrow F\otimes \Sym^j(\A(n,\R)^*)_\C$, $0\le j\le d$, such that
\begin{align}
	\label{eq:polynomialDecomp}
	\mu(f+\ell)=\sum_{j=0}^dY_j(f)[\ell],
\end{align}
where we identify $\Sym^j(\A(n,\R)^*)_\C$ with the space of $j$-homogeneous complex-valued polynomials on $\A(n,\R)$.\\
We denote by $\P_d\VConv(\R^n,F)$ the space of all continuous polynomial valuations $\mu:\Conv(\R^n,\R)\rightarrow F$ of degree at most $d$. 

Let $\P_d\VConv_k(\R^n,F)$ denote the subspace of all $k$-homogeneous valuations, that is, all $\mu\in\P_d\VConv(\R^n,F)$ satisfying $\mu(tf)=t^k\mu(f)$ for $t\ge 0$, $f\in \Conv(\R^n,\R)$. The following decomposition was obtained in \cite{KnoerrUlivelliPolynomialvaluationsconvex2026} based on results by Pukhlikov and Khovanski\u i \cite{PukhlikovKhovanskiiFinitelyadditivemeasures1992}.
\begin{theorem}[\cite{KnoerrUlivelliPolynomialvaluationsconvex2026}*{Theorem 1.3}]
	\label{theorem:homogeneousDecompositionPVConv}
	\begin{align*}
		\P_d\VConv(\R^n,F)=\bigoplus_{k=0}^{n+d}\P_d\VConv_k(\R^n,F)
	\end{align*}
\end{theorem}
Moreover, the functionals $Y^j$ from Eq.~\eqref{eq:polynomialDecomp} behave well respect to this decomposition.
\begin{proposition}[\cite{KnoerrUlivelliPolynomialvaluationsconvex2026}*{Proposition 3.2}]
	\label{proposition:translativeDecompCompatibleHomDecomp}
	For $\mu\in \P_d\VConv_k(\R^n,F)$ let $Y_j:\Conv(\R^n,\R)\rightarrow F\otimes\Sym^j(\A(n,\R)^*)_\C$ be the unique functions satisfying Eq.~\eqref{eq:polynomialDecomp}. Then $Y_j\in \P_{d-j}\VConv_{k-j}(\R^n,F\otimes \Sym^j(\A(n,\R)^*)_\C)$.
\end{proposition}

Given $\mu\in\P_d\VConv_k(\R^n,F)$, we may define its polarization $\bar{\mu}:\Conv(\R^n,\R)\rightarrow F$ by
\begin{align*}
	\bar{\mu}(f_1,\dots,f_k)=\frac{1}{k!}\frac{\partial^k}{\partial\lambda_1\dots\partial\lambda_k}\Big|_0\mu\left(\sum_{j=1}^k\lambda_jf_j\right)
\end{align*}
for $f_1,\dots,f_k\in\Conv(\R^n,\R)$. Then $\bar{\mu}$ is additive in each argument (compare \cite[Corollary 3.5 ]{KnoerrUlivelliPolynomialvaluationsconvex2026} and \cite[Section 4.2]{Knoerrsupportduallyepi2021}), which can be used to extend $\bar{\mu}$ to a distribution.
	\begin{theorem}[\cite{KnoerrUlivelliPolynomialvaluationsconvex2026}*{Theorem 1.4 and Corollary 4.8}]
	\label{theorem:GW}
	Let $F$ be a locally convex vector space and denote its completion by $\bar{F}$. For every $\mu\in\P_d\VConv_k(\R^n,F)$ there exists a unique distribution $\GW(\mu):C^\infty_c((\R^n)^k)\rightarrow \bar{F}$ such that for every $\lambda\in F'$, the distribution $\lambda(\GW(\mu))$ has compact support and satisfies
	\begin{align*}
		\lambda(\GW(\mu))[f_1\otimes\dots\otimes f_k]=\lambda\left(\bar{\mu}(f_1,\dots,f_k)\right)
	\end{align*}
	for all $f_1,\dots,f_k\in\Conv(\R^n,\R)\cap C^\infty(\R^n)$.
\end{theorem}
For $\phi_1,\dots,\phi_k\in C^\infty_c(\R^n)$, the value $\GW(\mu)[\phi_1\otimes\dots\otimes \phi_k]$ may be calculated in the following way, compare \cite[Section 5.1]{Knoerrsupportduallyepi2021}: Choose convex functions $f_1,\dots,f_k\in \Conv(\R^n,\R)$ such that $h_j:=f_j+\phi_j$ is convex for every $1\le j\le k$. Then
\begin{align}
	\label{equation:formulaGWElementaryTensor}
	\begin{split}
		&\GW(\mu)[\phi_1\otimes\dots\otimes \phi_k]\\
		&\quad =\sum\limits_{j=0}^k(-1)^{k-j}\frac{1}{(k-j)!j!}\sum_{\sigma\in S_k}\bar{\mu}(h_{\sigma(1)},\dots,h_{\sigma(j)},f_{\sigma(j+1)},\dots,f_{\sigma(k)}),
	\end{split}
\end{align}
where $S_k$ denotes the symmetric group.\\
\begin{remark}
	\label{remark:distributionsClosure}
	Eq.~\eqref{equation:formulaGWElementaryTensor} shows that $\GW(\mu))[\phi_1\otimes\dots\otimes \phi_k]$ belongs to $F$ even if $F$ is not complete. More generally, it is not difficulty to see that $\GW(\mu)[\phi]$ belongs to the sequential closure of $F$ in its completion for every $\phi\in C_c^\infty((\R^n)^k)$, however, we will not need this fact.
\end{remark}

For a semi-norm $|\cdot|_F$ on $F$, we set for $K\subset\Conv(\R^n,\R)$ and $\mu\in\P_d\VConv(\R^n,F)$,
\begin{align*}
	\|\mu\|_{F;K}=\sup_{f\in K}|\mu|_F\in [0,\infty]
\end{align*}
If the semi-norm is continuous, then this defines a semi-norm on $\P_d\VConv(\R^n,F)$. If $F$ is a locally convex vector space, we equip $\P_d\VConv(\R^n,F)$ with the topology induced by these semi-norms for all continuous semi-norms on $F$ and $K\subset\Conv(\R^n,\R)$ compact.

The following was shown for $d=0$ and continuous semi-norms in \cite[Section 5.1]{Knoerrsupportduallyepi2021}. The argument holds verbatim for $d>0$ and arbitrary semi-norms.
\begin{lemma}
	\label{lemma:EstimateGW}
	There exists a constant $C_k>0$ and a compact subset $K\subset \Conv(\R^n,\R)$ independent of $\mu\in \P_d\VConv_k(\R^n,F)$ such that for every semi-norm $|\cdot|_F$ on $F$
	\begin{align*}
		|\GW(\mu)[\phi_1\otimes\dots\otimes \phi_k]|_F\le C_k		\|\Psi\|_{F;K}\prod_{j=1}^k\|\phi_j\|_{C^2(\R^n)}
	\end{align*}
	for all $\phi_1,\dots,\phi_j\in C^\infty_c(\R^n)$.
\end{lemma}

For $\mu\in\P_d\VConv(\R^n,F)$ let $\mu=\sum_{k=0}^{n+d}\mu_k$ be the decomposition into its homogeneous components. We define the \emph{support of $\mu$} by
\begin{align*}
	\supp \mu:= \bigcup_{k=1}^{n+d}\Delta_k^{-1}(\supp\GW(\mu_k)),
\end{align*}
where $\Delta_k:\R^n\rightarrow (\R^n)^k$, $x\mapsto (x,\dots,x)$, denotes the diagonal embedding. Note that this is a closed subset of $\R^n$. Moreover, it follows from \cite[Theorem 4.6]{KnoerrUlivelliPolynomialvaluationsconvex2026} that $\mu$ has compact support if $F$ admits a continuous norm.

\begin{proposition}[\cite{KnoerrUlivelliPolynomialvaluationsconvex2026}*{Proposition 4.9}]
	\label{proposition:characterizationSupport} The support of $\mu\in \P_d\VConv(\R^n,F)$ is the unique minimum (with respect to inclusion) among all closed sets $A\subset \R^n$ with the following property: If $f,h\in \Conv(\R^n,\R)$ are two functions with $f\equiv h$ on a neighborhood of $A$, then $\mu(f)=\mu(h)$.
\end{proposition}
\begin{remark}
	The uniqueness of the minimum is only implicit in the statement in \cite{KnoerrUlivelliPolynomialvaluationsconvex2026}*{Proposition 4.9}, however, it is shown explicitly in the proof.
\end{remark}

For a closed set $A\subset\R^n$, let $\P_d\VConv_A(\R^n,F)$ denote the subspace of valuations with support contained in $A$. The following was shown for $d=0$ in \cite[Proposition 6.8]{Knoerrsupportduallyepi2021}. The proof holds verbatim for $d>0$. 
\begin{proposition}
	\label{proposition:semiNormsCompactlySupportedValuation}
	Let $A\subset\R^n$ be compact and convex and $|\cdot|_F$ a continuous semi-norm on $F$. Then for $s>0$,
	\begin{align*}
		\|\mu\|_{F;A,s}:=\sup\left\{|\mu(f)|_F: f\in\Conv(\R ^n,\R), \sup_{x\in A+B_s(0)}|f(x)|\le 1 \right\}
	\end{align*}
	defines a continuous semi-norm on $\P_d\VConv_A(\R^n,F)$. Moreover, $\|\mu\|_{F;A,s}$ is a norm if $|\cdot|_F$ is a norm, and the topology induced by these semi-norms (for all continuous semi-norms on $F$) coincides with the subspace topology of $\P_d\VConv_A(\R^n,F)$.
\end{proposition}
The following follows with the same argument as \cite[Corollary 6.9]{Knoerrsupportduallyepi2021}.
\begin{corollary}
	\label{corollary:EquivalenceSemiNorms}
	Let $A\subset\R^n$ be compact and convex. For $0<s<t$,
	\begin{align*}
		\|\mu\|_{F;A,t}\le \|\mu\|_{F;A,s}\le \left(\frac{2}{s}(2t+\diam A)+1\right)^k\|\mu\|_{F;A,t}
	\end{align*}
	for every $k$-homogeneous $\mu\in \P_d\VConv_A(\R^n,F)$.
\end{corollary}
Note that this implies that the semi-norms with $s=1$ already generate the topology.

\section{Topologies on spaces of measure-valued functionals}
\label{section:topologyFunctionals}

In this section we investigate three topologies on spaces of measure-valued functionals induced respectively by pointwise convergence, uniform convergence on relatively compact, or uniform convergence on bounded subsets of the space of continuous functions with compact support. Since the arguments do not depend on the specific properties of the underlying spaces, we consider the following more general framework.\\

Let $Y$ a locally compact and second countable Hausdorff space. For $A\subset Y$ compact, let $C_A(Y)$ denote the Banach space of all continuous functions $\phi:Y\rightarrow \C$ supported on $A$ equipped with the supremum norm. We consider $C_c(Y)$ as the (strict) inductive limit of the spaces $C_A(Y)$, $A\subset Y$ compact, that is, a set $U\subset C_c(Y)$ is open if and only if $U\cap C_A(Y)$ is open in $C_A(Y)$.
We identify the space of complex Radon measures $\M(Y)$ on $Y$ with the topological dual $C_c(Y)'$ of $C_c(Y)$. This space can be equipped with several topologies induced by uniform convergence on different families of bounded subsets. In our case, the bounded subsets of $C_c(Y)$ admit a rather simple description. The following is not difficult to show directly, however, it is also a general property of strict inductive limits, compare \cite{HorvathTopologicalvectorspaces1966}*{Chapter 2, §12}.
\begin{lemma}
	\label{lemma:BoundedSetsCompactSuppFunct}
	If $B\subset C_c(Y)$ is bounded, then there exists a compact subset $A\subset Y$ such that $B\subset C_A(Y)$.
\end{lemma}
Unless stated otherwise, we will consider $\M(Y)$ as a topological vector space with respect to the weak* topology.\\

For a metric space $X$, let $C(X,\M(Y))$ denote the space of all continuous maps $\Psi:X\rightarrow \M(Y)$. We will use the following conventions: For $\Psi\in C(X,\M(Y))$ and $x\in X$, $\Psi(x)\in \M(Y)$ denotes the value of $\Psi$ in $x\in X$. For $\phi\in C_c(Y)$, we denote by $\Psi(x;\phi):=\int_{Y}\phi d[\Psi(x)]$ the value of the corresponding integration functional.\\

The strongest topology  on $C(X,\M(Y))$ that we will consider is the \emph{compact-to-bounded topology}, which is the topology of uniform convergence on compact subsets with respect to the strong topology on $\M(Y)$. Due to the description of bounded sets in \autoref{lemma:BoundedSetsCompactSuppFunct}, this topology is induced by the family of semi-norms
\begin{align}
	\label{eq:defSeminormsBoundedFunctionals}
	\|\Psi\|_{A;K}:=\sup_{x\in K}|\Psi(x)|_{C_A(Y)'}
\end{align} 
for all compact $K\subset X$ and compact $A\subset Y$. Here, $|\cdot|_{C_A(Y)'}$ denotes the operator norm of an element of $\M(Y)$ considered as an element of the dual space $C_A(Y)'$. Note that these reduce to the semi-norms in Eq.~\eqref{eq:DefSemiNormBoundedTop} for $Y=\R^n$, $X=\Conv(\R^n,\R)$. The next result shows that these semi-norms are actually well defined.
\begin{lemma}
	\label{lemma:weakContImpliesBounded}
	For $\Psi\in C(X,\M(Y))$, $\|\Psi\|_{A;K}<\infty$ for all compact $A\subset Y$ and compact $K\subset X$.
\end{lemma}
\begin{proof}
	By definition, $\Psi(x)$, $x\in K$, is a family of bounded linear operators on $C_A( Y)$ with
	\begin{align*}
		\sup_{x\in K}|\Psi(x;\phi)|<\infty
	\end{align*}
	for every $\phi\in C_A(Y)$, since $\Psi$ is continuous with respect to the weak* topology and $K$ is compact. Since $C_A(Y)$ is a Banach space, the principle of uniform boundedness implies that there exists $C>0$ such that $|\Psi(x)|_{C_A(Y)'}\le C$ for all $x\in K$, which shows the claim. 
\end{proof}

	Let $\mathcal{B}$ be a family of bounded subsets in $C_c(Y)$ that cover $C_c(Y)$. \autoref{lemma:weakContImpliesBounded} implies that $C(X,\M(Y))$ becomes a locally convex vector space with respect to the family of semi-norms
	\begin{align*}
		\|\Psi\|_{(B;K)}:=\sup_{x\in K,\phi\in B}|\Psi(x;\phi)|
	\end{align*}
	for $K\subset X$ compact and $B\in\mathcal{B}$. We will focus on the following three cases.
\begin{definition}
	\begin{enumerate}
		\item If $\mathcal{B}$ is the family of all finite subsets, we call the induced topology the compact-to-weak* topology.
		\item If $\mathcal{B}$ is the family of all relatively compact subsets, we call the induced topology the compact-to-compact topology.
		\item If $\mathcal{B}$ is the family of all bounded subsets, we call the induced topology the compact-to-bounded topology.
	\end{enumerate}
\end{definition}
Again, \autoref{lemma:BoundedSetsCompactSuppFunct} shows that the compact-to-bounded topology is already induced by the semi-norms in Eq.~\eqref{eq:defSeminormsBoundedFunctionals}.

\subsection{Boundedness in $C(X,\M(Y))$}

\begin{lemma}\label{lemma:boundedSetsEquivalent}
	Let $B\subset C(X,\M(Y))$. The following are equivalent:
	\begin{enumerate}
		\item $B$ is bounded in the compact-to-weak* topology.
		\item $B$ is bounded in the compact-to-compact topology.
		\item $B$ is bounded in the compact-to-bounded topology.
	\end{enumerate}
\end{lemma}
\begin{proof}
	We only need to show that (1) implies (3). For a compact set $K\subset X$ and $A\subset Y$ compact, the set
	\begin{align*}
		\{\Psi(x)|_A: \Psi\in B, x\in K\}\subset C_A(Y)'
	\end{align*}
	is by assumption bounded in the weak* topology, so since $C_A(Y)$ is a Banach space, the principle of uniform boundedness implies that this set is bounded with respect to the operator norm. Thus $B$ is bounded in the compact-to-bounded topology.
\end{proof}

\begin{corollary}\label{corollary:characterizationContLinearMaps}
	The following are equivalent for a linear map $\lambda:C(X,\M(Y))\rightarrow\C$:
	\begin{enumerate}
		\item $\lambda$ is continuous with respect to the compact-to-weak* topology.
		\item $\lambda$ is continuous with respect to the compact-to-compact topology.
		\item $\lambda$ is continuous with respect to the compact-to-bounded topology.
	\end{enumerate}
\end{corollary}
\begin{proof}
	Since a linear map on a locally convex vector space is continuous if and only if it is bounded, the result follows directly from \autoref{lemma:boundedSetsEquivalent}.
\end{proof}

\begin{remark}
	From now on, we will for brevity call a subset in $C(X,\M(Y))$ bounded if it is bounded with respect to one of the three topologies in \autoref{lemma:boundedSetsEquivalent}.
\end{remark}

\begin{corollary}
	\label{corollary:EquivalenceTopologyBoundedSets}
	The compact-to-weak* and compact-to-compact topology coincide on bounded subsets of $C(X,\M(Y))$.
\end{corollary}
\begin{proof}
	Let $B\subset C(X,\M(Y))$ be bounded. Translating $B$ if necessary, it is sufficient to show that for every relatively compact $C\subset C_c(Y)$, compact $K\subset X$, and $\epsilon>0$, the set 
	\begin{align*}
		\{\Psi\in B:\|\Psi\|_{(C;K)}<\epsilon\}
	\end{align*}
	contains a neighborhood of $0$ (in $B$) in the compact-to-weak* topology. First note that since $C$ is in particular bounded, there exists $A\subset Y$ compact with $C\subset C_A(Y)$. As $B$ is bounded in the compact-to-bounded topology, there exists $M>0$ such that $\|\Psi\|_{A;K}\le M$ for every $\Psi\in B$. Fix $\epsilon>0$. Since $C$ is relatively compact, we find $N$ and $\phi_1,\dots,\phi_N\in C$ such that $C\subset \bigcup_{j=1}^N B_{\epsilon/2M}(\phi_j)$, where $B_{\epsilon/2M}(\phi_j)\subset C_A( Y)$ denotes the ball with radius $\epsilon/2M$ centered at $\phi_j$. For $\phi\in C$ we pick $\phi_j$ with $\|\phi-\phi_j\|_\infty\le \epsilon/2M$. Given $\Psi\in B$, we have
	\begin{align*}
		\sup_{x\in K}|\Psi(x;\phi)|\le \sup_{x\in K}|\Psi(x;\phi_j)|+\sup_{x\in K}|\Psi(x;\phi-\phi_j)|\le \sup_{x\in K}|\Psi(x;\phi_j)|+\epsilon/2.
	\end{align*}
	Since the set of all $\Psi\in B$ with $\sup_{x\in K}|\Psi(x;\phi_j)|<\epsilon/2$ for $1\le j\le N$ is open in $B$ with respect to the compact-to-weak* topology, this shows the desired result.
\end{proof}

\subsection{Completeness and quasi-completeness}
	\begin{lemma}
	\label{lemma:MeasureValuedFunctionals_QuasiComplete}
	$C(X,\M(Y))$ is quasi-complete with respect to the compact-to-weak* and compact-to-compact topology.
\end{lemma}
\begin{proof}
	We have to show that every bounded Cauchy net in $C(X,\M(Y))$ converges. Since both topologies coincide on bounded sets by \autoref{corollary:EquivalenceTopologyBoundedSets}, we may restrict ourselves to Cauchy nets $(\Psi_\alpha)_\alpha$ with respect to the compact-to-weak* topology.\\
	
	For a compact subset $A\subset Y$ and $x\in X$, $\Psi_\alpha(x)\in C_A( Y)'$ is a Cauchy net with respect to weak* convergence that is bounded in the operator norm due to \autoref{corollary:EquivalenceTopologyBoundedSets}. In particular, the pointwise limit (on $C_A(Y)$), which we denote by $\Psi(x)_A$, exists and belongs to $C_A( Y)'$.
	Note that for $A\subset B$ compact, the restriction of $\Psi(x)_B$ to $A$ is equal to $\Psi(x)_A$. Thus, there exists $\Psi(x)\in\M(Y)$ with $\Psi(x)|_A=\Psi(x)_A$ for every compact $A\subset Y$. We obtain a well defined map $\Psi:X\rightarrow\M(Y)$ and claim that $(\Psi_\alpha)_\alpha$ converges to $\Psi$ with respect to the compact-to-weak* topology. In order to see this, let $\phi\in C_c(Y)$ be fixed and $K\subset X$ be compact. Then for every $x\in K$, 
	\begin{align*}
		|\Psi_\alpha(x;\phi)-\Psi(x;\phi)|=\lim\limits_{\beta}|\Psi_\alpha(x;\phi)-\Psi_\beta(x;\phi)|.
	\end{align*} 
	Let $\epsilon>0$ be given. Since $(\Psi_\alpha)_\alpha$ is a Cauchy net with respect to the compact-to-weak* topology, there exists an index $\beta_0$ such that for $\alpha,\beta\ge\beta_0$
	\begin{align*}
		|\Psi_\alpha(x;\phi)-\Psi_\beta(x;\phi)|\le \epsilon \quad\text{for all}~x\in K.
	\end{align*}
	In particular $\sup_{x\in K}|\Psi_\alpha(x;\phi)-\Psi(x;\phi)|\le \epsilon$ for $\alpha\ge\beta_0$, which shows that $(\Psi_\alpha)_\alpha$ converges to $\Psi$ in the compact-to-weak* topology.\\
	
	It remains to see that $\Psi:X\rightarrow \M(Y)$ is continuous. We will argue by contradiction. Suppose that $\Psi$ is not continuous. Since $X$ is a metric space, $\Psi$ is not sequentially continuous, so we find $\epsilon>0$, $\phi\in C_c( Y)$, and a sequence $(x_j)_j$ in $X$ converging to $x\in X$ such that
	\begin{align*}
		|\Psi(x_j;\phi)-\Psi(x;\phi)|\ge \epsilon \quad \text{for every}~j\in\mathbb{N}.
	\end{align*}
	Note that $K:=\{x_j:j\in\mathbb{N}\}\cup\{x\}$ is compact. Since $(\Psi_\alpha)_\alpha$ converges uniformly to $\Psi$ on $K$ with respect to the compact-to-weak* topology, there exists an index $\alpha_0$ such that for $\alpha>\alpha_0$,
	\begin{align*}
		\epsilon\le |\Psi(x_j;\phi)-\Psi(x;\phi)|\le |\Psi_\alpha(x_j;\phi)-\Psi_\alpha(x;\phi)|+\frac{\epsilon}{2}\quad \text{for every}~j\in\mathbb{N}.
	\end{align*}
	Thus for $\alpha>\alpha_0$, $\epsilon/2\le |\Psi_\alpha(x_j;\phi)-\Psi_\alpha(x;\phi)|$ for every $j\in\mathbb{N}$ which contradicts the continuity of $\Psi_\alpha$ with respect to the weak* topology. In particular, $\Psi$ has to be continuous.
\end{proof}

\begin{lemma}
	\label{lemma:completeUniformStrongTop}
	$C(X,\M(Y))$ is complete with respect to the compact-to-bounded topology.
\end{lemma}
\begin{proof}
	Let $(\Psi_\alpha)_\alpha$ be a Cauchy net in $C(X,\M(Y))$ with respect to the compact-to-bounded topology. Then for every $x\in X$ and every compact $A\subset Y$, $(\Psi_\alpha(x))_\alpha$ is a Cauchy net in $C_A( Y)'$ with respect to the operator norm, and thus converges to an element $\Psi(x)_A\in C_A( Y)'$, since $C_A( Y)'$ is complete. Obviously, these limits are compatible under restrictions, so we obtain $\Psi(x)\in \M(Y)$ with $\Psi(x)|_A=\Psi(x)_A$ for every $A\subset Y$ compact. We claim that $(\Psi_\alpha)_\alpha$ converges to $\Psi$ with respect to the compact-to-bounded topology. Let $K\subset X$ and $A\subset Y$ be compact. Then for $\phi\in C_A( Y)$, $x\in K$,
	\begin{align*}
		\|\Psi(x;\phi)-\Psi_\alpha(x;\phi)|=\lim\limits_{\beta}\|\Psi_\alpha(x;\phi)-\Psi_\beta(x;\phi)|\le \limsup_{\beta}\|\Psi_\alpha-\Psi_\beta\|_{A;K}.
	\end{align*}
	Let $\epsilon>0$ be given. Since $(\Psi_\alpha)_\alpha$ is a Cauchy net with respect to the compact-to-bounded topology, there exists $\alpha_0$ such that $\|\Psi_\alpha-\Psi_\beta\|_{A;K}\le \epsilon$ for all $\alpha,\beta\ge\alpha_0$. Thus for $\alpha\ge \alpha_0$,
	\begin{align*}
		\|\Psi-\Psi_\alpha\|_{A;K}\le \epsilon.
	\end{align*}
	Thus $(\Psi_\alpha)_\alpha$ converges to $\Psi$. With the same arguments as in the proof of \autoref{lemma:MeasureValuedFunctionals_QuasiComplete}, one easily verifies that $\Psi$ is continuous.
\end{proof}

\subsection{Group actions}\label{section:topologyFunctionals:GroupAction}
	From now on we will assume that both $X$ and $Y$ carry a continuous group action by a locally compact topological group $G$, i.e. such that the maps 
	\begin{align*}
		&G\times X\rightarrow X && G\times Y\rightarrow Y
	\end{align*}
	induced by the group action are continuous. We denote the action of $g\in G$ on $x\in X$ by $g\cdot x$ and similarly for the action of $G$ on $Y$. Note that the action of $G$ on $Y$ induces a well-defined action on $C_c(Y)$ given by $[g\cdot \phi](y):=\phi(g^{-1}\cdot y)$ for $\phi\in C_c(Y)$, $y\in Y$. More generally, we can twist this action by equipping the vector bundle $Y\times\C$ with the structure of a nontrivial equivariant line bundle over $G$, i.e. by letting $G$ act on this bundle by 
	\begin{align*}
		g\cdot (y,z)=(g\cdot y, f(g,y)z)\quad\text{for}~(y,z)\in Y\times\C, g\in G, 
	\end{align*}
	where $f:G\times Y\rightarrow\C\setminus\{0\}$ is a continuous function satisfying 
	\begin{align*}
		&f(e,y)=1, &&f(hg,y)=f(h,g\cdot y)f(g,y).
	\end{align*}
	In this case, the action of $G$ on $C_c(Y)$ is given by
	\begin{align*}
		[g\cdot \phi](y):=f(g^{-1},y)\phi(g^{-1}\cdot y),\quad y\in Y, g\in G.
	\end{align*}
	We refer to \cite{KnoerrTranslationinvariantarea2026} for a natural case of such an action. in the following, we assume that $X$ and $Y\times \C$ are equipped with an action of this type.\\
		
	The following is a direct consequence of the continuity of $G\times Y\rightarrow Y$ and \autoref{lemma:BoundedSetsCompactSuppFunct} in combination with the fact that $Y$ is locally compact.
	\begin{lemma}\label{lemma:relativelyCompactSetsUnderGroupAction}
		For $U\subset G$ relatively compact and $B\subset C_c(Y)$ set 
		\begin{align*}
			U\cdot B:=\{g\cdot \phi:\phi\in B\}.
		\end{align*}
		Then the following holds:
		\begin{enumerate}
			\item If $B$ is bounded, so is $U\cdot B$.
			\item If $B$ is relatively compact, so is $U\cdot B$.
		\end{enumerate}
	\end{lemma}

		We obtain a representation $\pi$ on $C(X,\M(Y))$ in the following way: For $g\in G$ and $\Psi\in C(X,\M(Y))$, we define $\pi(g)\Psi\in C(X,\M(Y))$ by
	\begin{align*}
		[\pi(g)\Psi](x;\phi)=\Psi(g^{-1}\cdot x;g^{-1}\cdot\phi).
	\end{align*}
	for $x\in X$, $\phi\in C_c(Y)$. It is easy to check that $\pi(g)\Psi$ is well defined. We will be interested in the continuity properties of this map.

	\begin{lemma}\label{lemma:groupActsByContOperators}
		For every $g\in G$, the map $\pi(g):C(X,\M(Y))\rightarrow C(X,\M(Y))$ is continuous with respect to the compact-to-weak*, compact-to-compact, and compact-to-bounded topology.
	\end{lemma}
	\begin{proof}
		For $K\subset X$ compact and a bounded subset $B\subset C_c(Y)$, we have for $\Psi\in C(X,\M(Y))$
		\begin{align}
			\label{eq:EstimateOperatorNormContinuosAction}
			\sup_{x\in K,\phi\in B}|\pi(g)\Psi(x;\phi)|\le \sup_{x\in g^{-1}\cdot K,\phi\in g^{-1}\cdot B}|\Psi(x;\phi)|
		\end{align}
		If $B\subset C_c(Y)$ is finite, or relatively compact, or bounded, then \autoref{lemma:BoundedSetsCompactSuppFunct} shows that $g^{-1}\cdot B$ has the same properties. Since $G$ acts continuously on $X$, $g^{-1}\cdot K$ is compact as well, so Eq.~\eqref{eq:EstimateOperatorNormContinuosAction} shows that $\pi(g)$ is continuous with respect to these three topologies.
	\end{proof}
	
	\begin{lemma}
		\label{lemma:ActionContInUniformWeakTopology}
		Let $\Psi\in C(X,\M(Y))$. The map
		\begin{align}
			\label{eq:OrbitMap}
			\begin{split}
				G&\rightarrow C(X,\M(Y))\\
				g&\mapsto \pi(g)\Psi
			\end{split}
		\end{align}
		is continuous in the compact-to-weak* and compact-to-compact topology.
	\end{lemma}
	\begin{proof}
		First note that for $\Psi\in C(X,\M(Y))$ and $U\subset G$ compact, the set
		\begin{align*}
			\{\pi(g)\Psi:g\in U\}\subset C(X,\M(Y))
		\end{align*}
		is bounded in the compact-to-bounded topology: If $\phi\in C_c(Y)$, then $C_U:=\{g^{-1}\cdot\phi: g\in U\}$ is bounded in $C_c(Y)$ due to \autoref{lemma:relativelyCompactSetsUnderGroupAction} and for $K\subset X$ compact,
		\begin{align*}
			K_U:=\{g^{-1}\cdot x:x\in K,g\in U\}
		\end{align*}
		is compact in $K$ by the continuity of the action of $G$ on $K$, so
		\begin{align*}
			\sup_{x\in K}|\pi(g)\Psi(x;\phi)|\le \sup_{x\in K_U,\psi\in C_U}|\Psi(x;\psi)|=\|\Psi\|_{(C_U;K_U)}<\infty.
		\end{align*}
		Thus, the image of the restriction of the map in Eq.~\eqref{eq:OrbitMap} to relatively compact sets in $G$ is contained in a bounded set in $C(X,\M(Y))$. By \autoref{corollary:EquivalenceTopologyBoundedSets} it is thus sufficient to check that the map is continuous in the compact-to-weak* topology. By \autoref{lemma:groupActsByContOperators}, it is sufficient to show that the map is continuous at the neutral element of $G$.\\
		
		Let $U\subset G$ be a compact neighborhood of the neutral element. For $A\subset Y$ compact, the set $U^{-1}\cdot A=\{g^{-1}\cdot y:y\in A, g\in U\}$ is compact due to the continuity of the action of $G$ on $Y$. For $\phi\in C_A(Y)$ and $K\subset X$ compact, we have for $x\in K$ and $g\in U$,
		\begin{align*}
			&|\Psi(g^{-1}\cdot x;g^{-1}\cdot\phi)-\Psi(x;\phi)|\\
			\le& |\Psi(g^{-1}\cdot x;g^{-1}\cdot\phi)-\Psi(g^{-1}\cdot x;\phi)|+|\Psi(g^{-1}\cdot x;\phi)-\Psi(x;\phi)|\\
			\le& \|\Psi\|_{U^{-1}\cdot A;K_U}\|\phi-g^{-1}\cdot\phi\|_\infty+|\Psi(g^{-1}\cdot x;\phi)-\Psi(x;\phi)|,
		\end{align*}
		where $\|\Psi\|_{U^{-1}A,K_U}<\infty$ by the previous discussion. Since $\phi$ is continuous and has compact support, it is uniformly continuous, and we may choose a compact neighborhood $W$ of the neutral element in $G$ such that $\|\phi-g^{-1}\cdot\phi\|_\infty<\frac{\epsilon}{2\|\Psi\|_{U^{-1}\cdot A;K_U}}$ for all $g\in W$. Since $G$ acts continuously on $X$ and since $\Psi$ is continuous with respect to the weak* topology, the function
		\begin{align*}
			W\times K&\rightarrow \R\\
			(g,x)&\mapsto|\Psi(g^{-1}\cdot x;\phi)-\Psi(x;\phi)|
		\end{align*}
		is continuous. As $W\times K$ is compact, it is therefore uniformly continuous, and we find a neighborhood $\tilde{W}$ of the neutral element such that for all $g\in \tilde{W}$ and $x\in K$,
		\begin{align*}
			|\Psi(g^{-1}\cdot x;\phi)-\Psi(x;\phi)|<\frac{\epsilon}{2}.
		\end{align*}
		In total, we obtain for $g\in \tilde{W}$,
		\begin{align*}
			|\Psi(g^{-1}\cdot x;g^{-1}\cdot\phi)-\Psi(x;\phi)|<\epsilon.
		\end{align*}
		Thus the map in \eqref{eq:OrbitMap} is continuous with respect to the compact-to-weak* topology.
	\end{proof}

	\begin{proposition}\label{proposition:contRepCompactToCompact}
		The representation of $G$ on $C(X,\M(Y))$ is continuous with respect to the compact-to-compact topology.
	\end{proposition}
	\begin{proof}
		For $C\subset C_c(Y)$ relatively compact, $K\subset X$ compact, we have
		\begin{align*}
			&\|\pi(g_1)\Psi_1-\pi(g_2)\Psi_2\|_{(C;K)}\\
			\le &\|\pi(g_1)\Psi_1-\pi(g_2)\Psi_1\|_{(C;K)}+\|\pi(g_2)\Psi_1-\pi(g_2)\Psi_2\|_{(C;K)}
		\end{align*}
		for all $g_1,g_2\in G$, $\Psi_1,\Psi_2\in C(X,\M(Y))$. The first term on the right hand side is continuous in $g_1$ due to \autoref{lemma:ActionContInUniformWeakTopology}. For $\Psi_1$ fixed and a given $\epsilon>0$, we may therefore find a compact neighborhood $U$ of $g_1$ such that the first term is smaller than $\epsilon/2$ for all $g_2\in U$. For the second term, note that
		\begin{align*}
			C_U=\{g^{-1}\cdot\phi_2:\phi\in C, g\in U\}
		\end{align*}
		is relatively compact in $C_c(Y)$ by \autoref{lemma:relativelyCompactSetsUnderGroupAction}. As before,
		\begin{align*}
			K_U=\{g^{-1}\cdot x:x\in K,g\in U\}
		\end{align*}
		is compact in $X$, so for $g_2\in U$,
		\begin{align*}
			\|\pi(g_2)\Psi_1-\pi(g_2)\Psi_2\|_{(C;K)}\le \|\Psi_1-\Psi_2\|_{(C_U;K_U)}.
		\end{align*}
		As we fixed  $\Psi_1$, this shows that 
		\begin{align*}
			\|\pi(g_1)\Psi_1-\pi(g_2)\Psi_2\|_{(C;K)}<\epsilon
		\end{align*}
		for all $g_2\in U$ and $\Psi_2\in C(X,\M(Y))$ belonging to an open neighborhood of $\Psi_1$ with respect to the compact-to-compact topology. The claim follows.
	\end{proof}

\subsection{Smooth and strongly continuous vectors}\label{section:topologyFunctionals:Smooth}
	If the group $G$ from the previous section is a Lie group, then we can consider the corresponding spaces of smooth vectors of the induced representation on $C(X,\M(Y))$. However, since we are considering $C(X,\M(Y))$ with different topologies, this leads to three notions of smoothness which are a priori distinct. We will see in the following that this distinction is not necessary in our situation.\\
	
	Let $F$ be a locally convex vector space, $U\subset\R^n$ open, and $f:U\rightarrow F$ any map. We call $f$
	\begin{enumerate}
		\item \emph{strongly $C^k$} if $f$ is $k$-times continuously differentiable with respect to the topology on $F$.
		\item \emph{weakly $C^k$} if for every $\lambda\in F'$, the map $\lambda\circ f: U\rightarrow \C$ is $k$-times continuously differentiable.
		\item strongly or weakly $C^\infty$ if it is strongly or weakly $C^k$ for every $k\in\mathbb{N}$ respectively.
	\end{enumerate}
	For $k=0$, we will also call a weakly or strongly $C^0$ function weakly continuous or strongly continuous respectively.

	\begin{remark}
		Let us point out that in general a weakly $C^k$ function is not a strongly $C^k$ function. This is even the case if $F$ is equipped with the weak topology since a strongly $C^k$ function has partial derivatives that actually belong to $F$, whereas there does not need to exist any element in $F$ corresponding to the partial derivatives of the maps $\lambda\circ f$, $\lambda\in F'$. Nevertheless, under certain assumptions on $F$, these notions are closely related, as the following result shows.
	\end{remark}
	
	\begin{theorem}[\cite{GarrettModernanalysisautomorphic2018}*{Theorem~15.1.1}]\label{theorem:differentiabilityInterval}
		If $F$ is a quasi-complete locally convex vector space, then a weakly $C^k$ $F$-valued function on an interval is strongly $C^{k-1}$.
	\end{theorem}
	The following is a simple consequence of this result. Its proof is a minor modification of the arguments in  \cite{GarrettModernanalysisautomorphic2018}*{Section~15}.
	\begin{corollary}\label{corollary:weaklySmoothIsSmooth}
		Let $U\subset\R^n$ be an open set, $V$ a quasi-complete locally convex vector space, and $f:U\rightarrow V$. 
		\begin{enumerate}
			\item If $f$ is weakly $C^1$, then $f$ is strongly continuous.
			\item $f$ is weakly $C^\infty$ if and only if $f$ is strongly $C^\infty$.
		\end{enumerate}
	\end{corollary}
	\begin{proof}
		Note that (2) follows by induction from (1) since \autoref{theorem:differentiabilityInterval} shows that the partial derivatives of $f$ exist and are weakly $C^1$ in this case. Let us thus show that (1) holds. For $\lambda\in V'$, $u,v\in U$, $u$ fixed,
		\begin{align*}
			\frac{1}{|u-v|}\lambda\left(f(u)-f(v)\right)
		\end{align*}
		is uniformly bounded in $|u-v|\le \epsilon$ for some $\epsilon>0$ since $\lambda\circ f$ is continuously differentiable by assumption. Thus for $u\in U$ fixed, the set
		\begin{align*}
			\left\{\frac{1}{|v|}\left[f(u)-f(v)\right]: v\in U,|u-v|\le \epsilon\right\}
		\end{align*}
		is weakly bounded in $V$ and therefore bounded since $V$ is locally convex. If $N$ is a balanced and convex neighborhood of $0$, we therefore find $t>0$ such that this set is contained in $tN$, so $f(u)-f(v)\in |u-v|tN$ for all $v\in U$ with $|u-v|\le \epsilon$. Thus $f$ is strongly continuous in $u\in U$.
	\end{proof}

	\begin{corollary}\label{corollary:EquivalenceNotionsSmoothness}
		 For $\Psi\in C(X,\M(Y))$, consider the map
		 \begin{align}
		 	\label{eq:OrbitGroupActionMeasureFunc}
		 	\begin{split}
		 		G&\rightarrow C(X,\M(Y))\\
		 		g&\mapsto \pi(g)\Psi.
		 	\end{split}
		 \end{align}
		The following are equivalent.
		\begin{enumerate}
			\item This map is strongly $C^\infty$ in the compact-to-weak* topology.
			\item This map is strongly $C^\infty$ in the compact-to-compact topology.
			\item This map is strongly $C^\infty$ in the compact-to-bounded topology.
		\end{enumerate}
	\end{corollary}
	\begin{proof}
		By \autoref{corollary:weaklySmoothIsSmooth}, the three statements are equivalent to this map being weakly $C^\infty$ with respect to the corresponding topological dual spaces. Since the topological duals of $C(X,\M(Y))$ with respect to the three topologies are identical (as sets) by \autoref{corollary:characterizationContLinearMaps}, the claim follows.
	\end{proof}

	\begin{definition}\label{definition:SmoothVectorsStronglyContVectors}
		We call $\Psi\in C(X,\M(Y))$
		\begin{enumerate}
			\item smooth if the map in Eq.~\eqref{eq:OrbitGroupActionMeasureFunc} is strongly $C^\infty$ with respect to the compact-to-weak*, or compact-to-compact, or compact-to-bounded topology,
			\item strongly continuous if the map in Eq.~\eqref{eq:OrbitGroupActionMeasureFunc} is continuous with respect to the compact-to-bounded topology.
		\end{enumerate}
	\end{definition}
	Note that any smooth vector in $C(X,\M(Y))$ is automatically strongly continuous due to \autoref{corollary:EquivalenceNotionsSmoothness}.	
	\begin{proposition}\label{proposition:stronglyContVectorsClosed}
		The subspace of strongly continuous vectors is closed in $C(X,\M(Y))$ with respect to the compact-to-bounded topology.
	\end{proposition}
	\begin{proof}
		Let $\Psi\in C(X,\M(Y))$ belong to the closure of this space. It suffices to show that $g\mapsto \pi(g)\Psi$ is continuous at the neutral element. For a bounded set $B\subset C_c(Y)$, $K\subset X$ compact, we have for a compact neighborhood $U$ of the neutral element, $g\in U$, and $\tilde{\Psi}\in C(X,\M(Y))$,
		\begin{align*}
			&\|\pi(g)\Psi-\Psi\|_{(B;K)}\\
			\le& \|\pi(g)\Psi-\pi(g)\tilde{\Psi}\|_{(B;K)}+\|\pi(g)\tilde{\Psi}-\tilde{\Psi}\|_{(B;K)}+\|\tilde{\Psi}-\Psi\|_{(B;K)}\\
			\le & \|\Psi-\tilde{\Psi}\|_{(U^{-1}\cdot B;K_U)}+\|\pi(g)\tilde{\Psi}-\tilde{\Psi}\|_{(B;K)}+\|\tilde{\Psi}-\Psi\|_{(B;K)}.
		\end{align*}
		Here, $U^{-1}\cdot B$ is bounded by \autoref{lemma:BoundedSetsCompactSuppFunct}.	Since $\Psi$ belongs to the closure, we may pick a strongly continuous $\tilde{\Psi}\in C(X,\M(Y))$ such that the first and the last term are smaller than $\epsilon/3$ for some given $\epsilon>0$. For this choice of $\tilde{\Psi}$, we may pick a neighborhood $W\subset U$ of the neutral element on which the second term is smaller than $\epsilon/3$, since $\tilde{\Psi}$ is strongly continuous. The claim follows.
	\end{proof}
	The following follows with the same argument.
	\begin{corollary}\label{corrollary:StronglyContRepresentation}
		The representation of $G$ on the subspace of strongly continuous vectors in $C(X,\M(Y))$ is continuous with respect to the compact-to-bounded topology.
	\end{corollary}

	We conclude this section with two density results concerning the smooth elements in $C(X,\M(Y))$. Since the representation of $G$ on $C(X,\M(Y))$ is not necessarily continuous with respect to the three different topologies, we need a slightly weaker notion. Let us call a representation of $G$ on a locally convex vector space $V$ separately continuous if the map
	\begin{align*}
		G\times V&\rightarrow V\\
		(g,v)&\mapsto g\cdot v
	\end{align*}
	is separately continuous.  The following result is well known, however, it is usually stated for continuous representations although only the weaker notion is required. We refer to \cite{GarrettModernanalysisautomorphic2018}*{Theorem~14.6.1} for a proof that only requires a separately continuous action.
	\begin{corollary}\label{corollary:approxSmoothVectors}
		Let $V$ be a separately continuous representation of a Lie group $G$ on a quasi-complete topological vector space. If $(\phi_j)_j\subset C_c^\infty(G)$ is a smooth approximate identity and $v\in V$, then
		\begin{align*}
			v_j:=\int_{G}\phi_j(g)\pi(g)v dg
		\end{align*}
		converges to $v$, where the integral on the right hand side is understood in the Gelfand--Pettis sense.
	\end{corollary}
	
	This has the following consequence for our representation.
	\begin{corollary}
		\begin{enumerate}
			\item Smooth vectors are sequentially dense in $C(X,\M(Y))$ with respect to the compact-to-weak* and compact-to-compact topology. The same holds for any $G$-invariant subspace of $C(X,\M(Y))$ that is closed in the respective topology.
			\item Assume that $W\subset C(X,\M(Y))$ is $G$-invariant and closed with respect to the compact-to-bounded topology. If $W$ only contains continuous vectors, then smooth vectors are sequentially dense in $W$ with respect to the compact-to-bounded topology.
			\item The closure of the space of smooth vectors with respect to the compact-to-bounded topology coincides with the space of all strongly continuous vectors.
		\end{enumerate}
	\end{corollary}
	\begin{proof}
		The first two claims are a direct consequence of \autoref{corollary:approxSmoothVectors} since the representation is separately continuous with respect to these topologies by \autoref{lemma:groupActsByContOperators}, \autoref{lemma:ActionContInUniformWeakTopology}, and \autoref{corrollary:StronglyContRepresentation}. The last claim follows with the same reasoning from \autoref{corrollary:StronglyContRepresentation}, noting that every smooth vector is strongly continuous by \autoref{corollary:EquivalenceNotionsSmoothness} and that the space of strongly continuous vectors is closed with respect to the compact-to-bounded topology, compare \autoref{proposition:stronglyContVectorsClosed}.
	\end{proof}

\section{Local functionals on convex functions}\
	\label{section:localFunctionals}
	Let us briefly introduce some notation that will be used in the following sections. For a local functional $\Psi:\Conv(\R^n,\R)\rightarrow\M(\R^n)$ and $\phi\in C_c(\R^n)$, we denote by $\Psi[\phi]:\Conv(\R^n,\R)\rightarrow\C$ the functional defined by
	\begin{align*}
		\Psi[\phi](f):=\int_{\R^n}\phi(x)d\Psi(f;x).
	\end{align*}
	Given $\psi\in C(\R^n)$, we define $\psi\bullet \Psi:\Conv(\R^n,\R)\rightarrow\M(\R^n)$ by
	\begin{align*}
		[\psi\bullet \Psi](f;B)=\int_{B}\psi(x)d\Psi(f;x)
	\end{align*}
	for $f\in\Conv(\R^n,\R)$, $B\subset\R^n$ bounded Borel set. It is then easy to see that $\psi\bullet \Psi$ is a local functional. Moreover, if $\Psi$ is
	\begin{itemize}
		\item continuous, or
		\item a polynomial local functional of degree at most $d$, or
		\item homogeneous of degree $k\in\mathbb{N}$,
	\end{itemize}
	then so is $\psi\bullet \Psi$. In particular, this equips $\P_d\LV(\R^n)$ with the structure of a module over $C(\R^n)$.

\subsection{Relation to valuations on convex functions}

We begin our investigation of local functionals with the proof of \autoref{maintheorem:LocalFunctionalValuation}.

\begin{proof}[Proof of \autoref{maintheorem:LocalFunctionalValuation}]
	Assume that $f,h\in\Conv(\R^n,\R)$ are given such that $f\wedge h$ is convex. We need to show that
	\begin{align*}
		\Psi(f\vee h;\phi)+\Psi(f\wedge h;\phi)=\Psi(f;\phi)+\Psi(h;\phi)
	\end{align*}
	for every $\phi\in C_c(\R^n)$. Consider the open sets
	\begin{align*}
		&U:=\mathrm{int}\{x\in\R^n: f(x)\ge h(x)\}, &&V:=\mathrm{int}\{x\in\R^n:f(x)\le h(x)\}.
	\end{align*}
	If $\phi\in C_c(\R^n)$ is supported on $U$, then $\Psi(f\vee h;\phi)=\Psi(f;\phi)$, $\Psi(f\wedge h;\phi)=\Psi(h;\phi)$ since $\Psi$ is locally determined, which implies
	\begin{align*}
		\Psi(f\vee h;\phi)+\Psi(f\wedge h;\phi)=\Psi(f;\phi)+\Psi(h;\phi).
	\end{align*}
	The same reasoning applies if $\phi$ is supported on $V$. If $U\cup V=\R^n$, this implies the desired result using a continuous partition of unity subordinate to this cover.\\
	If this is not the case, we repeat the argument for $\epsilon>0$ with the functions
	\begin{align*}
		&f_\epsilon:=\max(f,f\wedge h+\epsilon), &&h_\epsilon:=\max(h,f\wedge h+\epsilon),
	\end{align*}
	which converge to $f$ and $h$ for $\epsilon\rightarrow0$ respectively. Note that $f_\epsilon\wedge h_\epsilon=f\wedge h+\epsilon$ is convex and converges to $f\wedge h$ for $\epsilon\rightarrow0$. Similarly,  $f_\epsilon\vee h_\epsilon=\max(f\vee h, f\wedge h+\epsilon)$, which converges to $f\vee h$ for $\epsilon\rightarrow0$. We claim that the open sets
	\begin{align*}
		&U_\epsilon:=\mathrm{int}\{x\in\R^n: f_\epsilon(x)\ge h_\epsilon(x)\}, &&V_\epsilon:=\mathrm{int}\{x\in\R^n:f_\epsilon(x)\le h_\epsilon(x)\}
	\end{align*}
	cover $\R^n$. Let $x\in \R^n$ and assume without loss of generality that $f(x)\ge h(x)$. Since $f$ and $h$ are continuous, we may distinguish the following three cases:
	\begin{itemize}
		\item If $f(x)>h(x)+\epsilon $, then there is a neighborhood $W$ of $x$ on which this inequality holds, and so $f_\epsilon\equiv f\equiv h_\epsilon$ on $W$. Thus $x\in U_\epsilon$.
		\item If $h(x)+\epsilon > f(x) >h(x)$, then there is a neighborhood $W$ on which both inequalities hold, and so $f_\epsilon\equiv h+\epsilon\equiv h_\epsilon$ on $W$. Thus $x\in U_\epsilon$.
		\item If $f(x)=h(x)$, then there is a neighborhood $W$ of $x$ such that $|f(y)-h(y)|<\epsilon$ for $y\in W$. In particular, $f_\epsilon\equiv f\wedge h+\epsilon\equiv h_\epsilon$ on $W$, so $x\in U_\epsilon$.
	\end{itemize}
	The previous argument thus implies for $\phi\in C_c(\R^n)$ and every $\epsilon>0$,
	\begin{align*}
		\Psi(f_\epsilon\vee h_\epsilon;\phi)+\Psi(f_\epsilon\wedge h_\epsilon;\phi)=\Psi(f_\epsilon;\phi)+\Psi(h_\epsilon;\phi).
	\end{align*}
	Since $\Psi$ is continuous with respect to the weak* topology, we can take the limit $\epsilon\rightarrow0$ to obtain
	\begin{align*}
		\Psi(f\vee h;\phi)+\Psi(f\wedge h;\phi)=\Psi(f;\phi)+\Psi(h;\phi),
	\end{align*}
	which completes the proof.
\end{proof}
Note that the previous argument only implicitly uses that the functions are convex by exploiting the fact that finite convex functions are continuous. In particular, the same argument can be applied to continuous local functionals on more general spaces of continuous functions. The following example shows that the statement fails in general if the functional is not continuous.

\begin{example}
	\label{example:nonContLocaFunc}
	For $f\in\Conv(\R,\R)$, define $\Psi:\Conv(\R,\R)\rightarrow\M(\R)$ by \begin{align*}
		\Psi(f):= \begin{cases}
			\delta_0 & \text{if}~f~\text{is not differentiable in}~0,\\
			0 & \text{else}.
		\end{cases}
	\end{align*}
	Then $\Psi$ is a local functional which is not continuous. In addition, it is not a valuation: Consider the functions $f,h\in\Conv(\R,\R)$ given by
	\begin{align*}
		&f(x)=\begin{cases}
			-2x& x\le 0,\\
			-x& x>0,
		\end{cases}&&h(x)=\begin{cases}
		-x& x\le 0,\\
		0& x>0.
	\end{cases}
	\end{align*}
	Then $	f\wedge h(x)=-x$, so $\Psi(f\wedge h)=0$. However, since $\Psi(f)=\Psi(h)=\Psi(f\vee h)=\delta_0$, this implies
	\begin{align*}
		\Psi(f)+\Psi(h)\ne \Psi(f\vee h)+\Psi(f\wedge h),
	\end{align*}
	so $\Psi$ is not a valuation.
\end{example}

\subsection{Construction of local functionals with the differential cycle}
\label{section:ConstructionLocalFunc}

In this section, we extend a construction of local functionals originally introduced in \cite{KnoerrMongeAmpereoperators2024,KnoerrSmoothvaluationsconvex2024}. The construction relies on the following uniqueness result for integral currents on the cotangent bundle $T^*\R^n$ of $\R^n$, which was established by Fu \cite{FuMongeAmperefunctions.1989}, and we refer to \cite{FedererGeometricmeasuretheory1969} for a general background on currents. Let $\vol_n$ denote the standard volume form on $\R^n$ and $\pi:T^*\R^n\cong \R^n\times(\R^n)^*\rightarrow\R^n$ the projection onto the first factor.
\begin{theorem}[\cite{FuMongeAmperefunctions.1989}*{Theorem~2.0}]
	\label{theorem:characterization_differential_cycle}
	Let $f:\R^n\rightarrow\mathbb{R}$ be a locally Lipschitzian function. There exists at most one integral $n$-current $S$ on $T^*\R^n$ such that
	\begin{enumerate}
		\item $S$ is closed, i.e. $\partial S=0$,
		\item $S$ is Lagrangian, i.e. $S\llcorner \omega_s=0$, where $\omega_s$ denotes the natural symplectic form,
		\item $S$ is locally vertically bounded, i.e. $\supp S\cap \pi^{-1}(A)$ is compact for all $A\subset \R^n$ compact,
		\item $S(\phi(x,y)\pi^*\vol_n)=\int_{\R^n}\phi(x,df(x))d\vol_n(x)$ for all $\phi\in C^\infty_c(T^*\R^n)$.
	\end{enumerate}
	Note that the right hand side of the last equation is well defined due to Rademacher's theorem.
\end{theorem}
If such a current exists, the function $f$ is called Monge-Amp\`ere, the corresponding current is denoted by $D(f)$ (it is denoted by $[df]$ in \cite{FuMongeAmperefunctions.1989}) and is called the \emph{differential cycle of $f$}. For smooth function, the differential cycle is therefore given by integration over the graph of its differential. Note that any element of $\Conv(\R^n,\R)$ admits a differential cycle, compare \cite[Proposition~3.1]{FuMongeAmperefunctions.1989}. We will need the following four properties.

\begin{theorem}[\cite{FuMongeAmperefunctions.1989}*{Theorem 2.2}]
	\label{theorem_FU_support_Differential_cycle}
	If $f\in\Conv(\R^n,\R)$, then
	\begin{align*}
		\supp D(f)\subset \text{graph }\partial f:=\left\{(x,y)\in T^*\R^n: y\in\partial f(x)\right\}.
	\end{align*}
	In particular, given an open set $U\subset V$,
	\begin{align*}
		\supp D(f)\cap\pi^{-1}(U)\subset U\times B_{\mathrm{lip}(f|_U)}(0),
	\end{align*}
	where $\mathrm{lip}(f|_U)$ denotes the Lipschitz-constant of $f|_U$.
\end{theorem}

\begin{remark}
	\label{remark:DiffCycleLocal}
	It follows from the remarks in \cite[Section~2.1]{FuMongeAmperefunctions.1989} that $f=h$ on an open subset $U\subset\R^n$ implies $D(f)|_{\pi^{-1}(U)}=D(h)|_{\pi^{-1}(U)}$.
\end{remark}

\begin{lemma}[\cite{KnoerrSmoothvaluationsconvex2024}*{Lemma 4.8}]
	\label{lemma:massDifferentialCycle}
	For $f\in\Conv(\R^n,\R)$, 
	\begin{align*}
		M_{\pi^{-1}(U_R(0))}(D(f))\le2^n\omega_n \left(\sup_{|x|\le R+1}|f(x)|\right)^n.
	\end{align*}
\end{lemma}
\begin{theorem}[\cite{KnoerrSingularValuationsHadwiger2025}*{Theorem 4.6}]
	\label{theorem:joint_continuity_differential_cycle}
	Let $(\phi_j)_j\subset C(\R^n\times(\R^n)^*)$ be a sequence that converges locally uniformly to $\phi_0\in C(\R^n\times(\R^n)^*)$ and let $(f_j)_j$ be a sequence in $\Conv(\R^n,\R)$ converging to $f_0\in\Conv(\R^n,\R)$. If there exists a compact subset $A\subset {\R^n}$ such that $\supp\phi_j\subset\pi^{-1}(A)$ for all $j\in\mathbb{N}$, then
	\begin{align*}
		\lim\limits_{j\rightarrow\infty}D(f_j)[\phi_j\omega]=D(f_0)[\phi_0\omega]
	\end{align*}
	for any $\omega\in \Omega^n(\R^n\times(\R^n)^*)$.
\end{theorem}

In our case, we will need to consider differential forms that also depend on the values of the given functions.
\begin{theorem}
	\label{theorem:LocalFunctionalFromDifferentialCycle}
	For every $\tau\in C(\R\times \R^n\times(\R^n)^*,\Lambda^n(\R^n\times (\R^n)^*)^*)$ the following defines a continuous local functional $\Phi_\tau:\Conv(\R^n,\R)\rightarrow \M(\R^n)$:
	\begin{align*}
		\Phi_\tau(f;B):=D(f)[1_{\pi^{-1}(B)}(x,y)\tau(f(x),x,y)],
	\end{align*}
	$f\in\Conv(\R^n,\R)$, $B\subset\R^n$ bounded Borel set.
\end{theorem}
\begin{proof}
	It follows from \autoref{remark:DiffCycleLocal} that $\Phi_\tau$ is locally determined, so we only need to check that this map is continuous with respect to weak* convergence. Let $(f_j)_j$ be a sequence in $\Conv(\R^n,\R)$ converging to $f\in \Conv(\R^n,\R)$. Set $\tau_0(x,y)=\tau(f(x),x,y)$, $\tau_j(x,y)=\tau(f_j(x),x,y)$. Then $(\tau_j)_j$ converges locally uniformly to $\tau_0$ on $\R^n\times(\R^n)^*$. Given $\phi\in C_c(\R^n)$, we thus obtain
	\begin{align*}
		\Phi(f_j;\phi)=D(f_j)[\pi^*\phi \tau_j],
	\end{align*}
	where $\supp\phi\tau_j\subset \pi^{-1}\supp\phi$ for every $j\in\mathbb{N}$. \autoref{theorem:joint_continuity_differential_cycle} implies 
	\begin{align*}
		\lim\limits_{j\rightarrow\infty}\Phi(f_j;\phi)=D(f)[\pi^*\phi \tau_0]=\Phi(f;\phi).
	\end{align*}
	Thus $\Phi$ is continuous with respect to the weak* topology.
\end{proof}
The following provides a description of the kernel of this procedure.
\begin{theorem}
	\label{theorem:KernelDifferentialCycle}
	Let $\tau\in C(\R\times \R^n\times(\R^n)^*,\Lambda^n(\R^n\times (\R^n)^*)^*)$. Then $\Phi_\tau=0$ if and only if $\tau$ is a multiple of the symplectic form.
\end{theorem}
\begin{proof}
	If $\tau$ is a multiple of the symplectic form, then the same holds for differential form $\tau_f$ for every $f\in\Conv(\R^n,\R)$, and thus $D(f)[\pi^{*}\phi\tau_f]=0$ since $D(f)$ vanishes on multiples of the symplectic form.\\
	Now assume that $\Phi_\tau=0$. For simplicity we identify $(\R^n)^*\cong\R^n$ using the standard inner product. Obviously, it is sufficient to show that the differential form $\tau_{t_0}$ is a multiple of the symplectic form for every $t_0\in \R$. Fix $t_0\in \R$ and $(x_0,y_0)\in\R^n\times\R^n$. For an orthonormal basis $\mathcal{B}=(u_1,\dots,u_n)$ of $\R^n$ and $\lambda=(\lambda_1,\dots,\lambda_n)$, $\lambda_j>0$, we consider the convex function 
	\begin{align*}
		f_{\mathcal{B},\lambda}(x):= t_0+\langle x-x_0,y_0\rangle+\sum_{j=1}^n \frac{\lambda_j}{2} \langle x-x_0,u_j\rangle^2.
	\end{align*}
	Then $f_{\mathcal{B},\lambda}(x_0)=t_0$ and $df_{\mathcal{B},\lambda}(x_0)=y_0$, and the tangent space to the graph of $df_{\mathcal{B},\lambda}$ in $(x_0,y_0)$ is spanned by $(u_i,\lambda_i u_i)$. Since $\Phi_\tau=0$, we obtain for any $\phi\in C_c(\R^n)$,
	\begin{align*}
		0=\Phi_\tau(f_{\mathcal{B},\lambda};\phi)=D(f_{\mathcal{B},\lambda})[\pi^*\phi \tau_{f_{\mathcal{B},\lambda}}]=\int_{\R^n} \phi(x)F_{f_{\mathcal{B},\lambda}}^*\tau_{f_{\mathcal{B},\lambda}},
	\end{align*}
	where $F_{f_{\mathcal{B},\lambda}}:\R^n\rightarrow \R^n\times \R^n$, $x\mapsto (x,df_{\mathcal{B},\lambda}(x))$ is the graph map. As this holds for all $\phi\in C_c(\R^n)$, we see that the restriction of $\tau_{f_{\mathcal{B},\lambda}}$ to the tangent spaces of $\mathrm{graph}(df_{\mathcal{B},\lambda})$ vanishes. In particular, the restriction of $\tau_{t_0}$ to the subspace of $T_{(x_0,y_0)}(\R^n\times\R^n)$ spanned by $(u_i,\lambda u_i)$ vanishes. Since this holds for all choices of $\lambda_i>0$ and orthonormal bases $(u_1,\dots,u_n)$, it now follows from \cite{AbardiaEvequozEtAlFlagareameasures2019}*{Lemma 2.4} that $\tau_{t_0}$ is a multiple of the symplectic form.
\end{proof}
The following example will be used in the next section.
\begin{example}\label{example:MA}
	For $\phi\in C(\R\times\R^n\times(\R^n)^*)$, the differential form $\tau$ given by
	\begin{align*}
		\tau|_{(t,x,y)}:=\phi(t,x,y)dy_1\wedge\dots\wedge dy_n
	\end{align*}
	satisfies
	\begin{align*}
		\Phi_\tau(f;B)=\int_{B}\phi(f(x),x,df(x))\det(D^2f(x))dx
	\end{align*}
	for $f\in \Conv(\R^n,\R)\cap C^2(\R^n)$, $B\subset\R^n$ bounded Borel set. In particular, this local functional may be expressed in terms of the real Monge--Amp\`ere operator $\MA$ (compare the discussion after Eq.~\eqref{eq:definitionEabc})  if $\phi$ does not depend on the third variable. 
\end{example}

\subsection{The homogeneous decomposition for polynomial local functionals}
Recall that $\P_d\LV_k(\R^n)\subset \P_d\LV(\R^n)$ denotes the subspace of all $k$-homogeneous local functionals, i.e. the subspace of all $\Psi\in \P_d\LV(\R^n)$ such that $\Psi(tf)=t^k\Psi(f)$ for all $t\ge 0$, $f\in\Conv(\R^n,\R)$.
\begin{proof}[Proof of \autoref{maintheorem:HomDecompPdLV}]
	By \autoref{maintheorem:LocalFunctionalValuation}, every $\Psi\in \P_d\LV(\R^n)$ is a valuation. It thus follows from  \autoref{theorem:homogeneousDecompositionPVConv} that for every $f\in\Conv(\R^n,\R)$, the map $t\mapsto \Psi(tf)$ is a polynomial in $t\ge0$ of degree at most $n+d$, i.e. there exist unique functionals $Z_k:\Conv(\R^n,\R)\rightarrow \M(\R^n)$ that are homogeneous of degree $0\le k\le n+d$ such that for all $f\in\Conv(\R^n,\R)$ and $t\ge 0$,
	\begin{align}
		\label{equation:proofHomogeneousDecomp}
		\Psi(tf)=\sum_{k=0}^{n+d}t^kZ_k(f).
	\end{align}
	Plugging in $t=0,\dots, n+d$ and using the inverse of the Vandermonde matrix, we obtain constants $c_{kj}\in\R$ independent of $f\in\Conv(\R^n,\R)$ such that for every $f\in\Conv(\R^n,\R)$,
	\begin{align*}
		Z_k(f)=\sum_{j=0}^{n+d} c_{kj}\Psi(jf).
	\end{align*}
	Since $f\mapsto \Psi(jf)$ is a continuous local functional and polynomial of degree at most $d$ for every $0\le j\le n+d$, so is $Z_k$, i.e. $Z_k\in \P_d\LV_k(\R^n)$.  This finishes the proof.
\end{proof}

We will obtain \autoref{maintheorem:topDegree} from the following characterization.
\begin{theorem}[\cite{KnoerrMongeAmpereoperators2024}*{Theorem~4.8}]
	\label{theorem:characterizationTopDegree_InvariantCase}
	For every $\Psi\in \P_0\LV_{n}(\R^n)$ there exists a unique function $\psi\in C(\R^n)$ such that
	\begin{align*}
		\Psi(f;B)=\int_{B}\psi(x) d\MA(f;x)
	\end{align*}
	for all $f\in \Conv(\R^n,\R)$, $B\subset\R^n$ bounded Borel set. 	
\end{theorem}
\begin{proof}[Proof of \autoref{maintheorem:topDegree}]	
	\autoref{proposition:translativeDecompCompatibleHomDecomp} shows that there exist unique  functionals $Y_j\in \P_{d-j}\VConv_{n+d-j}(\R^n,\M(\R^n))\otimes\Sym^j(\A(n,\R)^*)_\C$ for $0\le j\le d$ such that
	\begin{align*}
		\Psi(f+\ell)=\sum_{j=0}^dY_j(f)[\ell], \quad f\in\Conv(\R^n,\R), \ell\in \A(n,\R).
	\end{align*}
	In particular, $Y_d\in \P_0\VConv_{n}(\R^n,\M(\R^n))\otimes \Sym^d(A(n,\R)^*)_\C$. Since 
	\begin{align*}
		\Psi(f+t\ell)=\sum_{j=0}^dt^jY_j(f)[\ell],
	\end{align*}
	we may use the inverse of the Vandermonde matrix to obtain $c_{ij}\in\R$ independent of $f$, $\ell$, and $\Psi$ such that
	\begin{align*}
		Y_i(f)[\ell]=\sum_{j=0}^{n+d}c_{ij}\Psi(f+j\ell).
	\end{align*}
	In particular, $Y_j$ is locally determined. Thus $Y_d\in \P_0\VConv_{n}(\R^n,\M(\R^n))\otimes \Sym^d(A(n,\R)^*)_\C$ is locally determined, and by applying \autoref{theorem:characterizationTopDegree_InvariantCase} to the different components, we obtain a continuous function $\psi\in C(\R^n)\otimes\Sym^d(A(n,\R)^*)_\C$ such that
	\begin{align*}
		Y_d(f;B)[\ell]=\int_{B} \psi(x)[\ell] d\MA(f;x)
	\end{align*}
	for every bounded Borel set $B\subset\R^n$, $f\in\Conv(\R^n,\R)$, $\ell\in \A(n,\R)$.	Consider the differential form $\tau\in C(\R\times \R^n\times(\R^n)^*,\Lambda^n(\R^n\times (\R^n)^*)^*)$ given by
	\begin{align*}
		\tau|_{(t,x,y)}=\psi(x)[t,y]dy_1\wedge\dots \wedge dy_n.
	\end{align*}
	Then $\Phi_\tau\in \P_d\LV_{n+d}(\R^n)$, compare \autoref{theorem:LocalFunctionalFromDifferentialCycle}, and using \autoref{example:MA} it is easy to see that $\Psi-\Phi_\tau\in \P_d\LV_{n+d}(\R^n)$ is a polynomial valuation of degree at most $d-1$ and homogeneous of degree $n+d$. Since the corresponding component in the homogeneous decomposition of $\P_{d-1}\LV(\R^n)$ vanishes by \autoref{maintheorem:HomDecompPdLV}, this implies $\Psi=\Phi_\tau$. This completes the proof.
\end{proof}

\subsection{Support and topology}
\label{section:supportTopology}
Since $\P_d\LV_k(\R^n)\subset \P_d\VConv_k(\R^n,\M(\R^n))$ by \autoref{maintheorem:LocalFunctionalValuation}, we may consider the Goodey--Weil distributions from \autoref{theorem:GW} associated to any such valuation. However, it will be more convenient to consider the following modification of this construction.
\begin{corollary}
	Let $1\le k\le n+d$. For every $\Psi\in \P_d\LV_k(\R^n)$ there exists a unique distribution $\widehat{\GW}(\Psi)$ on $(\R^n)^{k+1}$ such that
	\begin{align*}
		\widehat{\GW}(\Psi)[\phi_1\otimes\dots\otimes \phi_{k+1}]=\left(\GW(\Psi)[\phi_1\otimes\dots\otimes \phi_{k}]\right)[\phi_{k+1}].
	\end{align*}
	Moreover, there exists a constant $C>0$ and a compact subset $K\subset\Conv(\R^n,\R)$ independent of $\Psi$ such that for any compact subset $A\subset\R^n$ with $\supp\phi_{k+1} \subset A$,
	\begin{align}
		\label{eq:EstimateGW}
		|\widehat{\GW}(\Psi)[\phi_1\otimes\dots\otimes \phi_{k+1}]|\le C \|\Psi\|_{A;K}\prod_{j=1}^k\|\phi_j\|_{C^2(\R^n)} \|\phi_{k+1}\|_\infty.
	\end{align}
\end{corollary}
\begin{proof}
	Note that $|\cdot|_{C_A(\R ^n)'}$ is a semi-norm on $\M(\R^n)$. From \autoref{lemma:EstimateGW} we obtain a compact set $K\subset\Conv(\R^n,\R)$ and a constant $C>0$ such that
	\begin{align*}
		|(\GW(\Psi)[\phi_1\otimes\dots\otimes \phi_k])[\phi_{k+1}]|\le&|\GW(\Psi)[\phi_1\otimes\dots\otimes \phi_k]|_{C_A(\R^n)'}\|\phi_{k+1}\|_\infty\\
		\le& C \|\Psi\|_{A;K}\prod_{j=1}^k\|\phi_j\|_{C^2(\R^n)}\|\phi_{k+1}\|_\infty.
	\end{align*}
	The claim thus follows from the Schwartz Kernel Theorem, compare \cite{GaskproofSchwartzskernel1961}.
\end{proof}
\begin{remark}\label{remark:ExtendedGWZeroHom}
	For $k=0$ and $\Psi\in \P_d\LV_0(\R^n)$, we define 
	\begin{align*}
		\widehat{\GW}(\Psi):=\Psi(0)\in\M(\R^n).
	\end{align*}
\end{remark}

In the rest of this section, we will examine the support of these distributions in order to obtain a simpler description of the compact-to-bounded topology on $\P_d\LV(\R^n)$. 
\begin{lemma}
	\label{lemma:ConsequenceLocalitySupport}
	Let $\Psi\in \P_d\LV(\R^n)$, $\phi\in C_c(\R^n)$. Then the valuation $\Psi[\phi]\in \P_d\VConv(\R^n,\C)$ satisfies
	\begin{align*}
		\supp\Psi[\phi]\subset \supp\Psi\cap \supp\phi.
	\end{align*} 
\end{lemma}
\begin{proof}
	We may assume that $\Psi$ is $k$-homogeneous. For $k=0$, the valuation $\Psi[\phi]$ is constant, so its support is empty by definition. Thus assume that $1\le k\le n+d$. Let us consider the case that $\supp\Psi\cap \supp\phi=\emptyset$. We need to show that the support of $\Psi[\phi]$ is empty, i.e. that this valuation vanishes identically. If this is not the case, then there exist $\phi_1,\dots,\phi_k\in C^\infty_c(\R^n)$ with $\supp\phi_j\cap\supp\phi=\emptyset$ such that $\GW(\Psi[\phi])[\phi_1\otimes\dots\otimes\phi_k]\ne 0$. Let $U$ be a neighborhood of $\supp\phi$ such that $\supp\phi_j\subset \R^n\setminus U$ and fix a function $f\in\Conv(\R^n,\R)$ such that $f+\sum_{j=1}^k\lambda_j\phi_j$ is convex for all $\lambda_j\in [-1,1]$. Then this function coincides with $f$ on $U$, and thus
	\begin{align*} 		\GW(\Psi[\phi])[\phi_1\otimes\dots\otimes\phi_k]=&\frac{1}{k!}\frac{\partial^k}{\partial\lambda_1\dots\partial\lambda_k}\Big|_0\Psi\left(f+\sum_{j=1}^k\lambda_j\phi_j;\phi\right)\\
		=&\frac{1}{k!}\frac{\partial^k}{\partial\lambda_1\dots\partial\lambda_k}\Big|_0\Psi\left(f;\phi\right)= 0
	\end{align*}
	since $\Psi$ is locally determined, which is a contradiction. Thus $\Psi[\phi]=0$ if $\supp\Psi\cap \supp\phi=\emptyset$.\\
	
	In the general case, let $f,h\in\Conv(\R^n,\R)$ be two functions such that $f\equiv h$ on a neighborhood $U$ of $\supp\Psi\cap \supp\phi$. Since $\supp\phi$ is compact, we find a continuous function $\psi\in C_c(\R^n)$ with $\supp \psi\subset U$, $0\le \psi\le 1$, and $\psi\equiv 1$ on $\supp\Psi\cap\supp\phi$. Then $\Psi[(1-\psi)\phi]=0$ by the previous discussion, so since $f$ and $h$ coincide on the neighborhood $U$ of $\supp \psi\phi$ and as $\Psi$ is locally determined, we obtain
	\begin{align*}
		\Psi[\phi](f)=\Psi[\psi\phi](f)=\Psi[\psi\phi](h)=\Psi[\phi](h).
	\end{align*} 
	\autoref{proposition:characterizationSupport} thus implies $\supp\Psi[\phi]\subset \supp\Psi\cap \supp\phi$.
\end{proof}

\begin{corollary}\label{corollary:SupportEvaluationContainedInSupport}
	Let $1\le k\le n+d$ and $\Psi\in \P_d\LV_k(\R^n)$. Then $\supp\Psi(f)\subset \supp\Psi$ for every $f\in\Conv(\R^n,\R)$.
\end{corollary}
\begin{proof}
	If $\phi\in C_c(\R^n)$ satisfies $\supp\phi\cap \supp\Psi=\emptyset$, then the support of $\Psi[\phi]\in\P_d\VConv_k(\R^n,\C)$ is empty by \autoref{lemma:ConsequenceLocalitySupport}, and so $\Psi[\phi]=0$ since $k\ne 0$. In particular, $\Psi(f;\phi)=0$.
\end{proof}
\begin{remark}
	For $k=0$ this statement fails trivially, since the support of any $0$-homogeneous valuation is empty by definition. We introduce a modified notion of support below to circumvent this problem.
\end{remark}
Recall that $\Delta_{k}:\R^n\rightarrow(\R^n)^{k}$ denotes the diagonal embedding.
\begin{proposition}
	\label{proposition:supportExtendedGW}
	Let $1\le k\le n+d$. For $\Psi\in \P_d\LV_k(\R^n)$, the support of $\widehat{\GW}(\Psi)$ is equal to $\Delta_{k+1}(\supp\Psi)$. In particular, it is contained in the diagonal.
\end{proposition}
\begin{proof}
	Let us show that $\supp\widehat{\GW}(\Psi)\subset\Delta_{k+1}(\supp\Psi)$. We have to show that for $\phi_1,\dots,\phi_{k+1}\in C_c^\infty(\R^n)$ with $\bigcap_{j=1}^{k+1}\supp\phi_j\cap \supp\Psi=\emptyset$, we have $\widehat{\GW}(\Psi)[\phi_1\otimes\dots\otimes\phi_{k+1}]= 0$. By definition,
	\begin{align*}
		\widehat{\GW}(\Psi)[\phi_1\otimes\dots\otimes\phi_{k+1}]=\GW(\Psi[\phi_{k+1}])[\phi_1\otimes\dots\otimes \phi_k],
	\end{align*}
	so since $\Psi[\phi_{k+1}]$ is supported on $\supp\Psi\cap \supp\phi_{k+1}$ by \autoref{lemma:ConsequenceLocalitySupport}, the right hand side is equal to $0$.\\
	
	Let us show that $\Delta_{k+1}(\supp\Psi)\subset \supp\widehat{\GW}(\Psi)$. Fix $x\in \supp\Psi$. It is enough to show that for any open neighborhood $U$ of $x$ there exist functions $\phi_1,\dots,\phi_{k+1}\in C^\infty_c(\R^n)$ with $\supp \phi_j\subset U$ and $\widehat{\GW}(\Psi)[\phi_1\otimes\dots\otimes \phi_{k+1}]\ne 0$.\\
	First, since $x\in \supp\Psi$, we find functions $\phi_1,\dots,\phi_k$ with $\supp\phi_j\subset U$ for $1\le j\le k$ such that $\GW(\Psi)[\phi_1\otimes\dots\otimes \phi_k]\in\M(\R^n)$ is not the zero measure (note that this is indeed a measure by \autoref{remark:distributionsClosure}). We claim that this measure is supported on $U$. Choose functions $f_j\in \Conv(\R^n,\R)$ such that $f_j+t\phi_j$ is convex for $t\in [-1,1]$. Then
	\begin{align*}
		\GW(\Psi)[\phi_1\otimes\dots\otimes \phi_k]=\frac{1}{k!}\frac{\partial^k}{\partial\lambda_1\dots\partial\lambda_k}\Big|_0\Psi\left(\sum_{j=1}^kf_j+\lambda_j\phi_j\right).
	\end{align*}
	If $\psi\in C^\infty_c(\R^n)$ satisfies $\supp\psi\cap U=\emptyset$, then the functions $\sum_{j=1}^kf_j+\lambda_j\phi_j$ and $\sum_{j=1}^kf_j$ coincide on a neighborhood of $\supp\psi$ for every $\lambda_j\ge 0$. Since $\Psi$ is locally determined, this implies
	\begin{align*}
		\left(\GW(\Psi)[\phi_1\otimes\dots\otimes \phi_k]\right)[\psi]=0,
	\end{align*}
	so the measure $\GW(\Psi)[\phi_1\otimes\dots\otimes \phi_k]$ is supported on $U$. Since it is not the zero measure, we find $\phi_{k+1}\in C_c^\infty(\R^n)$ with $\supp\phi_{k+1}\subset U$ and
	\begin{align*}
		0\ne \left(\GW(\Psi)[\phi_1\otimes\dots\otimes \phi_k]\right)[\phi_{k+1}]=\widehat{\GW}(\Psi)[\phi_1\otimes\dots\otimes \phi_{k+1}].
	\end{align*}
	This completes the proof.
\end{proof}

For $\Psi\in \P_d\LV(\R^n)$, let $\Psi=\sum_{k=0}^{n+d}\Psi_k$ be the decomposition into its homogeneous components. We then set
\begin{align*}
	\supp_{\mathrm{loc}}\Psi:=\bigcup_{k=0}^{n+d}\Delta_k^{-1}(\supp \widehat{\GW}(\Psi_k))
\end{align*}
and call $\supp_{\mathrm{loc}}\Psi$ the \emph{local support} of $\Psi$. Note that this differs from the support $\supp\Psi$ of $\Psi$ considered as a valuation on $\Conv(\R^n,\R)$.
\begin{corollary}
	For $\Psi\in\P_d\LV(\R^n)$, the set $\supp_{\mathrm{loc}}\Psi$ is the unique minimum among all closed sets $A\subset\R^n$ that satisfy the following two properties:
	\begin{enumerate}
		\item For every $f\in\Conv(\R^n,\R)$, $\supp\Psi(f)\subset A$.
		\item If $f,h\in\Conv(\R^n,\R)$ satisfy $f\equiv h$ in a neighborhood of $A$, then
		\begin{align*}
			\Psi(f)=\Psi(h).
		\end{align*} 
	\end{enumerate}
\end{corollary}
\begin{proof}
	Using the homogeneous decomposition, we may assume that $\Psi$ is $k$-homogeneous. For $k=0$, $\Psi(f)=\Psi(0)$ for all $f\in\Conv(\R^n,\R)$, so $\supp_{\mathrm{loc}}\Psi=\supp\Psi(f)$ (compare \autoref{remark:ExtendedGWZeroHom}), which implies the claim.\\
	If $1\le k\le n+d$, then \autoref{proposition:supportExtendedGW} shows that $\supp\Psi=\supp_{\mathrm{loc}}\Psi$, and the claim follows from \autoref{corollary:SupportEvaluationContainedInSupport} and \autoref{proposition:characterizationSupport}. 
\end{proof}

Next, we are going to adapt the semi-norms in \autoref{proposition:semiNormsCompactlySupportedValuation} to our specific setting.
\begin{proposition}
	\label{proposition:semiNormslocalFunctional}
	For $A\subset\R^n$ compact and convex with non-empty interior and $\Psi\in\P_d\LV(\R^n)$, set
	\begin{align*}
		\|\Psi\|_{A,\delta}:=\sup\left\{|\Psi(f;\phi)|: \supp\phi\subset A,\|\phi\|_\infty\le 1, \sup_{x\in A+\delta B_1(0)}|f(x)|\le1\right\}.
	\end{align*}
	Then $\|\cdot\|_{A,\delta}$ is a continuous semi-norm on $\P_d\LV(\R^n)$ with respect to the compact-to-bounded topology and the topology generated by these semi-norms for all  compact and convex $A\subset\R^n$ and a fixed $\delta>0$ coincides with the compact-to-bounded topology. Moreover, for $0<s<t$  and $\Psi\in\P_d\LV_k(\R^n)$,
	\begin{align}
		\label{eq:relationLocalSeminorms}
		\|\Psi\|_{A,t}\le \|\Psi\|_{A,s}\le \left(\frac{2}{s}(2t+\diam A)+1\right)^k\|\Psi\|_{A,t}.
	\end{align}
\end{proposition}
\begin{proof}
	Let us first verify that $\|\cdot\|_{A,\delta}$ is a well-defined and continuous semi-norm. Let $\psi_A\in C_c(\R ^n)$ be a function with $\|\psi_A\|_\infty\le 1$, $\supp\psi_A\subset A+\delta B_1(0)$ and $\psi_A=1$ on $A$. Then $\Psi(f;\phi)=[\psi_A\bullet\Psi](f;\phi)$ for every $\phi\in C_c(\R^n)$ with $\supp\phi\subset A$, and the valuation $\psi_A\bullet\Psi$ is supported on $A+\delta B_1(0)$. In particular,
	\begin{align*}
		\|\Psi\|_{A,\delta}=&\sup\left\{ |\Psi(f)|_{C_{A}(\R^n)'}:\sup_{x\in A+\delta B_1(0)}|f(x)|\le1\right\}\\
		=&\sup\left\{ |[\psi_A\bullet\Psi](f)|_{C_{A}(\R^n)'}:\sup_{x\in A+\delta B_1(0)}|f(x)|\le1\right\}\\
		=&\|\psi_A\bullet \Psi\|_{C_{A}(\R^n)';A,\delta}
	\end{align*}
	for the continuous semi-norm defined in \autoref{proposition:semiNormsCompactlySupportedValuation}. In particular, Eq.~\eqref{eq:relationLocalSeminorms} is a direct consequence of \autoref{corollary:EquivalenceSemiNorms}. It is easy to see that the map $\Psi\mapsto \psi_A\bullet\Psi$ is continuous with respect to the compact-to-bounded topology on $\P_d\LV(\R^n)$, so $\|\cdot\|_{A,\delta}$ is a continuous semi-norm on $\P_d\LV(\R^n)$ with respect to this topology.\\
	It remains to see that these semi-norms generate the topology of $\P_d\LV(\R^n)$. Let $K\subset \Conv(\R^n,\R)$ and $A\subset \R^n$ be compact. We may assume that $A$ is convex and has non-empty interior. Since $\Psi$ is locally determined,
	\begin{align*}
		\|\Psi\|_{A;K}=\sup_{f\in K}|\Psi(f)|_{C_A(\R^n)'}=\sup_{f\in K}|[\psi_A\bullet\Psi](f)|_{C_A(\R^n)'}=\|\psi_A\bullet\Psi\|_{A;K},
	\end{align*}
	where $\psi_A\bullet\Psi$ is compactly supported on $A+\delta B_1(0)$. From \autoref{proposition:semiNormsCompactlySupportedValuation} (and \autoref{corollary:EquivalenceSemiNorms}) we obtain a constant $D>0$ depending on $K$, $A$, and $\delta$ only such that
	\begin{align*}
		\|\psi_A\bullet\Psi\|_{A;K}\le D\|\psi_A\bullet\Psi\|_{C_{A}(\R^n)';A,\delta}= D\|\Psi\|_{A,\delta}.
	\end{align*}
	Thus $\|\Psi\|_{A;K}\le D\|\Psi\|_{A,\delta}$ for every $\Psi\in \P_d\LV(\R^n)$, which shows that the semi-norms $\|\cdot\|_{A,\delta}$, $A\subset \R^n$ compact and convex with non-empty interior, generate the topology of $\P_d\LV(\R^n)$. 
\end{proof}

\begin{proof}[Proof of \autoref{maintheorem:Frechet}]
	Since $\P_d\LV(\R^n)$ is complete with respect to the compact-to-bounded topology by \autoref{lemma:completeUniformStrongTop}, we only have to check that the topology is generated by a countable family of continuous semi-norms. We claim that we can take the semi-norms $\|\cdot\|_{B_m(0),1}$, $m\in\mathbb{N}$. Due to \autoref{proposition:semiNormslocalFunctional}, it is sufficient to bound the semi-norms $\|\cdot\|_{A,\delta}$ for $A\subset \R^n$ compact and convex with non-empty interior, $\delta>0$, in terms of these semi-norms.\\
	Choose $m\in\mathbb{N}$ such that $A\subset B_m(0)$ and $R\ge \max(1,\delta)$ such that $B_m(0)\subset A+(R-1)B_1(0)$. From \autoref{proposition:semiNormslocalFunctional}, we obtain a constant $D(\delta,R)>0$ such that
	\begin{align*}
		\|\Psi\|_{A,\delta}\le D(\delta,R)\|\Psi\|_{A,R}.
	\end{align*}
	By definition
	\begin{align*}
		&\|\Psi\|_{A,R}\\
		=&\sup\left\{|\Psi(f;\phi)|: \supp\phi\subset A,\|\phi\|_\infty\le 1, \sup_{x\in A+R B_1(0)}|f(x)|\le1\right\}\\
		\le& \sup\left\{|\Psi(f;\phi)|: \supp\phi\subset B_m(0),\|\phi\|_\infty\le 1, \sup_{x\in B_m(0)+B_1(0)}|f(x)|\le1\right\}\\
		=&\|\Psi\|_{B_m(0),1}.
	\end{align*}
	Thus $\|\Psi\|_{A,\delta}\le D(\delta,R)\|\Psi\|_{B_m(0),1}$. Since the semi-norms $\|\cdot\|_{A,\delta}$, $A\subset \R^n$ compact and convex with non-empty interior, generate the compact-to-bounded topology by \autoref{proposition:semiNormslocalFunctional}, the claim follows.
\end{proof}

\subsection{Examples of strongly continuous or smooth polynomial local functionals}
The goal of this section is to illustrate the differences between the different notions of regularity from \autoref{section:topologyFunctionals:Smooth} for $\P_d\LV(\R^n)$. These will play a more prominent role in \cite{KnoerrIntegralrepresentationpolynomial2026}.\\
Note that the group $\Aff(n,\R)$ acts continuously on $\R^n$ and $\Conv(\R^n,\R)$, compare \autoref{lemma:continuityActionAffonConv}, so we obtain an operation on $C(\Conv(\R^n,\R),\M(\R^n))$ given by
\begin{align*}
	\pi(g)\Psi(f;\phi)=\Psi(f\circ g;\phi\circ g),
\end{align*}
which is precisely the operation considered in  \autoref{section:topologyFunctionals:GroupAction}. By construction, $\pi(g)\Psi$ is locally determined if $\Psi$ is locally determined. Moreover, $\pi(g)\in \P_d\LV(\R^n)$ if $\Psi$ is in addition polynomial of degree at most $d$, so we obtain a representation on $\P_d\LV(\R^n)$.\\

We will not need the representation of the full affine group and instead only consider the restriction to the space of translations. We denote the subspace of strongly continuous local functionals (with respect to the group of translations) by $\P_d\LV^0(\R^n)$ and the subspace of smooth local functionals by $\P_d\LV^\infty(\R^n)$, compare \autoref{section:topologyFunctionals:Smooth}.
\begin{corollary}
	\label{corollary:StronglyContCompatibilityOperations}
	 If $W\subset \P_d\LV(\R^n)$ is a finite dimensional translation invariant subspace, then $W\subset \P_d\LV^\infty(\R^n)$.
\end{corollary}
\begin{proof}
	Since $W$ is a finite dimensional continuous representation with respect to the compact-to-compact topology, every vector is smooth, compare \autoref{section:topologyFunctionals:Smooth}.
\end{proof}
The following example shows that $\P_d\LV^0_k(\R^n)$ is in general a proper subspace of $\P_d\LV_k(\R^n)$. For a $k$-dimensional subspace $E\subset\R^n$, let $\MA_E:\Conv(E,\R)\rightarrow\M(E)$ denote the Monge--Amp\`ere operator on $E$ and $i_E:E\rightarrow \R^n$ the inclusion map.
\begin{corollary}
	\label{corollary:ExampleNotStronglyContinuous}
	Let $1\le k\le n-1$. For $E\in \Gr_k(\R^n)$, the map
	\begin{align*}
		f\mapsto i_{E*}\MA_E(f|_E)
	\end{align*}
	defines an element of $\P_0\LV_k(\R^n)$ that is not strongly continuous.
\end{corollary}
\begin{proof}
	Obviously, the restriction $\Conv(\R^n,\R)\rightarrow\Conv(E,\R)$, $f\mapsto f|_E$, is continuous. Since $\MA_E:\Conv(E,\R)\rightarrow \M(E)$ is continuous and since the pushforward of measures along proper maps is compatible with weak* convergence, the map above is continuous. It is easy to check that it is locally determined, $k$-homogeneous, and polynomial of degree $0$. Thus it belongs to $\P_0\LV_k(\R ^n)$. Let us denote it by $\Psi$. In order to see that the map
	\begin{align*}
		\R^n&\rightarrow \LV_k(\R^n)\\
		x&\mapsto \pi(x)\Psi
	\end{align*}
	is not continuous with respect to the compact-to-bounded topology, consider the support function $f=h_{B_E(0)}\in\Conv(\R^n,\R)$ of the unit ball $B_E(0)\subset E$. Then $\Psi(h_{B_E(0)})=c \delta_0$ for some $c>0$, compare \cite{AleskerValuationsconvexfunctions2019}*{Lemma~2.4}. In particular, for $x\in E^\perp$ and $\phi\in C_c(\R^n)$, we have
	\begin{align*}
		&\left(\Psi-
		[\pi(x)\Psi]\right)(h_{B_E(0)};\phi)
		=\Psi(h_{B_E(0)};\phi)-
		\Psi(h_{B_E(0)}(\cdot+x);\phi(\cdot+x))\\
		&=\Psi(h_{B_E(0)};\phi)-
		\Psi(h_{B_E(0)};\phi(\cdot+x))=c[\phi(0)-\phi(x)]
	\end{align*}
	because $h_{B_E(0)}$ is invariant with respect to translations of its argument in $E^\perp$. Given $R>0$ and $x\ne 0$, we can always find a function $\phi\in C_c(\R^n)$ with $\supp\phi\subset B_R(0)$, $\|\phi\|_\infty=1$, $\phi(0)=1$ and $\phi(x)=0$, so in particular
	\begin{align*}
		\|\Psi-\pi(x)\Psi\|_{K,B_R(0)}\ge c \quad \text{for all}~x\in E^\perp\setminus\{0\}
	\end{align*}
	for $K=\{h_{B_E(0)}\}\subset\Conv(\R^n,\R)$. Thus $\Psi$ cannot be strongly continuous.
\end{proof}

Recall that for $\tau\in C(\R\times \R^n\times(\R^n)^*,\Lambda^n(\R^n\times (\R^n)^*)^*)$, we denote by $\Phi_\tau:\Conv(\R^n,\R)\rightarrow\M(\R^n)$ the map given by
\begin{align*}
	\Phi_\tau(f;B)=D(f)[1_{\pi^{-1}(B)}(x,y)\tau(f(x),x,y)],
\end{align*}
which is a continuous local functional by \autoref{theorem:LocalFunctionalFromDifferentialCycle}.

\begin{theorem}
	\label{theorem:SmoothLVWithDifferentialCycle}
	Let $\tau\in C(\R\times \R^n\times(\R^n)^*,\Lambda^n(\R^n\times (\R^n)^*)^*)$.
	\begin{enumerate}
		\item The map
		\begin{align*}
			\R^n&\rightarrow C(\Conv(\R^n,\R),\M(\R^n))\\
			z&\mapsto \pi(z)\Phi_\tau
		\end{align*}
		is continuous with respect to the compact-to-bounded topology.
		\item If $\tau\in C^\infty(\R\times \R^n\times(\R^n)^*,\Lambda^n(\R^n\times (\R^n)^*)^*)$, then this map is smooth with respect to the compact-to-bounded topology.
	\end{enumerate}
\end{theorem}
\begin{proof}
	We will show the second claim, the first follows with a similar argument. Note that any $\tau\in C ^\infty(\R\times \R^n\times(\R^n)^*,\Lambda^n(\R^n\times (\R^n)^*)^*)$ is a finite linear combination of differential forms of the form $\phi(t,x,y)\tau$, where $\tau\in \Lambda^n(\R^n\times (\R^n)^*)^*$ is translation invariant and $\phi\in C^\infty(\R\times \R^n\times(\R^n)^*)$. We claim that $z\mapsto \pi(z)\Psi_{\phi\tau}$ is differentiable with respect to the compact-to-bounded topology with differential
	\begin{align*}
		\frac{d}{dt}\pi(z+tv)\Psi_{\phi\tau}=\Psi_{(-\partial_v\phi(\cdot-z))\tau}.
	\end{align*}
	Since $\tau$ is translation invariant, it is easy to see that it is sufficient to consider the case $z=0$.	We have
	\begin{align*}
		&\left[\frac{\pi(v)\Psi_{\phi\tau}-\Psi_{\phi\tau}-\Psi_{(-\partial_v\phi)\tau}}{|v|}\right](f;\psi)\\
		=&D(f)\left[\psi(x) \left[\frac{\phi(f(x),x-v,y)-\phi(f(x),x,y)+\partial_v\phi(f(x),x,y)}{|v|}\right]\tau\right].
	\end{align*}
	Assume that $|v|\le 1$.	Let $K\subset \Conv(\R^n,\R)$ be a compact subset, $R>0$. First, there exists $C>0$ such that any $f\in K$ is bounded by $C$ on $B_{R+1}(0)$ by \autoref{proposition_compactness_Conv}. Thus  \autoref{proposition_convex_functions_local_lipschitz_constants} shows that the Lipschitz constants of $f\in K$ on $B_{R+1}(0)$ are bounded by some $L>0$.  Using \autoref{lemma:massDifferentialCycle} and \autoref{theorem_FU_support_Differential_cycle}, we obtain for $f\in K$, $\psi\in C_{B_R(0)}(\R^n)$ with $\|\psi\|_\infty\le1$, and $|v|\le 1$,
	\begin{align*}
		&\left|\left[\frac{\pi(v)\Psi_{\phi\tau}-\Psi_{\phi\tau}-\Psi_{(-\partial_v\phi)\tau}}{|v|}\right](f;\psi)\right|\\
		\le& 2^n\omega_n \left(\sup_{|x|\le R+1}|f(x)|\right)^n\|\tau\|_\infty\\
		&\times \sup_{|t|\le C,|x|\le R+1, |y|\le L }\left|\frac{\phi(t,x-v,y)-\phi(t,x,y)+\partial_v\phi(f(x),x,y)}{|v|}\right|.
	\end{align*}
	The right hand side of this inequality converges uniformly to $0$ for $f\in K$ since $\phi$ is a smooth function, which shows that $z\mapsto \pi(z)\Psi_{\phi\tau}$ is differentiable with respect to the compact-to-bounded topology. 		
\end{proof}

Recall that \autoref{maintheorem:topDegree} shows that every element in $\P_d\LV_{n+d}(\R^n)$ admits an integral representation as in \autoref{theorem:SmoothLVWithDifferentialCycle}. This implies the following result.
\begin{corollary}
	$\P_d\LV^0_{n+d}(\R^n)=\P_d\LV_{n+d}(\R^n)$	
\end{corollary}

\section{Classification results}
		\label{section:classificationResults}	
		\subsection{Translation invariant polynomial local functionals}
		The space $\P_0\LV(\R^n)^{tr}$ was described explicitly in \cite{KnoerrMongeAmpereoperators2024}. Following \cite{KnoerrMongeAmpereoperators2024}, we set $\MAVal(\R^n):=\P_0\LV(\R^n)^{tr}$, $\MAVal_k(\R^n):=\P_0\LV_k(\R^n)^{tr}$ for $0\le k\le n$. The main results of \cite{KnoerrMongeAmpereoperators2024} identify this space with a finite dimensional space of Monge--Amp\`ere type operators. Let $\mathrm{M}(n,k)\subset\Poly(\Sym^2(\R^n))$ denote the subspace spanned by the $k$-minors of symmetric $(n\times n)$-matrix. By definition $\mathrm{M}(n,0)$ consists of constant polynomials.		
			\begin{theorem}[\cite{KnoerrMongeAmpereoperators2024}*{Theorem 1.2}]
				\label{theorem:ClassificationMA}
				For a continuous map $\Psi:\Conv(\R^n,\R)\rightarrow\mathcal{M}(\R^n)$ the following are equivalent:
				\begin{enumerate}
					\item $\Psi\in\MAVal_k(\R^n)$.
					\item There exists a differential form $\tau\in\Lambda^{n-k}(\R^n)^*\otimes \Lambda^{k}((\R^n)^*)^*$ such that for all $f\in\Conv(\R^n,\R)\cap C^2(\R^n)$,
					\begin{align*}
						\Psi(f;B)=\int_{\pi^{-1}(B)\cap \mathrm{graph}(df)}\tau
					\end{align*}
					for all bounded Borel sets $B\subset\R^n$.
					\item There exists a polynomial $P_\Psi\in \mathrm{M}(n,k)$ such that for all $f\in\Conv(\R^n,\R)\cap C^2(\R^n)$,
					\begin{align*}
						\Psi(f;B)=\int_B P_\Psi(D^2f(x))d\vol_n(x)
					\end{align*}
					for all bounded Borel sets $B\subset \R^n$.
				\end{enumerate}
			\end{theorem}
			Note that the polynomial $P_\Psi\in \mathrm{M}(n,k)$ is uniquely determined by $\Psi\in \MAVal_k(\R^n)$. We thus obtain a well defined isomorphism of vector spaces
			\begin{align}
				\label{eq:interpretationMAValAsPolynomials}
				\MAVal(\R^n)\cong \mathrm{M}_n:=\bigoplus_{k=0}^n\mathrm{M}(n,k).
			\end{align}
			
			In contrast, the differential form in \autoref{theorem:ClassificationMA} is  not unique, as discussed in \autoref{theorem:KernelDifferentialCycle}. Let $P\Lambda^{n-k,k}\subset \Lambda^{n-k}(\R^n)^*\otimes \Lambda^k((\R^n)^*)^*$ denote the subspace of primitive differential forms, i.e. all $\tau$ such that $\omega_s\wedge\tau=0$, where $\omega_s$ denotes the natural symplectic form on $\R^n\times (\R^n)^*$. 			
			\begin{theorem}[\cite{KnoerrMongeAmpereoperators2024}*{Theorem~6.18}]\label{theorem:IsmorphismPrimitiveFormsMAVal}
				The map 
				\begin{align*}
					P\Lambda^{n-k,k}&\rightarrow\MAVal_k(\R^n)\\
					\tau&\mapsto \Phi_\tau
				\end{align*}
				is bijective.
			\end{theorem}
			\begin{corollary}
				\label{corollary:DefTheta}
				There exists a well defined and injective map
				\begin{align*}
					\Theta_d:\Poly_{d}(\A(n,\R))\otimes \MAVal(\R^n)\rightarrow \P_d\LV(\R^n)^{tr}
				\end{align*}
				such that for $P\in \Poly_d(\A(n,\R))$, $\Psi\in\MAVal(\R^n)$,
				\begin{align}
					\label{eq:FormulaTheta}
					[\Theta_d(P\otimes \Psi)](f;B)=\int_{B} P(f(x),df(x))P_\Psi(D^2f(x))dx
				\end{align}
				for all $f\in\Conv(\R^n,\R)\cap C^2(\R^n)$, $B\subset\R^n$ bounded Borel set.
			\end{corollary}
			\begin{proof}
				Set $P\Lambda^{n}:=\bigoplus_{k=0}^nP\Lambda^{n-k,k}$. We define $\Theta_d$ as the restriction of the map in \autoref{theorem:LocalFunctionalFromDifferentialCycle} to the subspace
				\begin{align*}
					\Poly_{d}(\A(n,\R))\otimes \MAVal(\R^n)\cong \Poly_{d}(\A(n,\R))\otimes P\Lambda^{n}\subset C(\R\times \R^n\times(\R^n)^*,P\Lambda^n),
				\end{align*}
				where we consider $\Poly_d(\A(n,\R))$ as a subset of $C(\R\times (\R^n)^*)$, which we identify with functions on $\R\times \R^n\times(\R^n)^*$ that do not depend on the second variable. Then $\Theta_d(Q)$ defines a continuous local functional for every $Q\in \Poly_d(\Aff(n,\R))\otimes \MAVal(\R^n)$ and satisfies Eq.~\eqref{eq:FormulaTheta}. It is easy to see that this implies $\Theta_d(Q)\in \P_d\LV(\R^n)^{tr}$ for every $Q\in \Poly_d(\Aff(n,\R))\otimes \MAVal(\R^n)$. The fact that $\Theta_d$ is injective follows from \autoref{theorem:KernelDifferentialCycle} and the Lefschetz decomposition, compare \cite{HuybrechtsComplexgeometry2005}*{Proposition~1.2.30}.
			\end{proof}
			
			\autoref{maintheorem:translationInvariantCase} follows from the following result using the isomorphism in Eq.~\eqref{eq:interpretationMAValAsPolynomials}.
			\begin{theorem}
				\label{theorem:ClassificationTranslationInvariant}
				\begin{align*}
					\Theta_d:\Poly_{d}(\A(n,\R))\otimes \MAVal(\R^n)\rightarrow \P_d\LV(\R^n)^{tr}
				\end{align*}
				is an isomorphism of $\GL(n,\R)$-modules.
			\end{theorem}
			We split the proof of \autoref{theorem:ClassificationTranslationInvariant} into two steps. Let $W_{d,l}\subset\P_d\LV(\R^n)^{tr}$ denote the subspace of all $\Psi\in\P_d\LV(\R^n)^{tr}$ such that
			\begin{align*}
				t\mapsto \Psi(f+t)
			\end{align*}
			is a polynomial of degree at most $0\le l\le d$ in $t\in\R$.
			\begin{lemma}
				\label{lemma:verticaltranslationInvariant}
				The map
				\begin{align*}
					\Theta_d:\Poly_d((\R^n)^*)\otimes\MAVal(\R^n)\mapsto W_{d,0}
				\end{align*}
				is an isomorphism of $\GL(n,\R)$-modules.
			\end{lemma}
			\begin{proof}
				It is easy to see that the map commutes with the action of $\GL(n,\R)$ on both spaces. Due to \autoref{corollary:DefTheta} it is thus sufficient to show that the map is surjective. We use induction on $d$, where the case $d=0$ holds by definition. Assume that the claim holds for all such functionals of degree at most $d-1$. If $\Psi\in W_{d,0}$, then for $\lambda\in (\R^n)^*$, we have
				\begin{align*}
					\Psi(f+\lambda)=\sum_{j=0}^dY_{j,\Psi}(f)[\lambda]
				\end{align*}
				for $f\in\Conv(\R^n,\R)$, where $Y_{j,\Psi}\in \P_{d-j}\LV(\R^n)\otimes \Sym^j(\R^n)$, compare \autoref{proposition:translativeDecompCompatibleHomDecomp}. Since $\Psi$ is translation invariant, we have for $x\in \R^n$, $\phi\in C_c(\R^n)$,
				\begin{align*}
					\sum_{j=0}^dY_{j,\Psi}(f(\cdot+x);\phi(\cdot+x))[\lambda(\cdot+x)]=&\Psi(f(\cdot+x)+\lambda(\cdot+x);\phi(\cdot+x))\\
					=&\Psi(f+\lambda;\phi)=\sum_{j=0}^dY_{j,\Psi}(f;\phi)[\lambda],
				\end{align*}
				so comparing the degrees of homogeneity in $\lambda$, we obtain
				\begin{align*}
					Y_{d,\Psi}(f(\cdot+x);\phi(\cdot+x))[\lambda]=Y_{d,\Psi}(f;\phi)[\lambda],
				\end{align*}
				i.e. $Y_{d,\Psi}\in \P_0\LV(\R^n)^{tr}\otimes \Sym^d(\R^n)$. Since $\P_0\LV(\R^n)^{tr}=\MAVal(\R^n)$, we find an element $Q\in \Sym^d(\R^n)\otimes \MAVal(\R^n)$ such that the highest order term of $\lambda\mapsto [\Theta_d(Q)] (\cdot+\lambda)$ coincides with $Y_{d,\Psi}$. Since $\Theta_d(Q)\in \P_d\LV(\R^n)^{tr}$, $\Psi-\Theta_d(Q)\in \P_{d-1}\LV(\R^n)^{tr}$, to which we can apply the induction hypothesis.
			\end{proof}
			\begin{proof}[Proof of \autoref{theorem:ClassificationTranslationInvariant}]
				It is again easy to see that the map commutes with the action of $\GL(n,\R)$. Since it is injective by \autoref{theorem:KernelDifferentialCycle}, we only need to show that it is surjective. We will show by induction on $0\le l\le d$ that $W_{d,l}$ is contained in the image. Since $W_{d,d}=\P_d\LV(\R^n)^{tr}$, this implies the desired result.\\
				The case $l=0$ is covered by \autoref{lemma:verticaltranslationInvariant}, so let us assume that $l\ge 1$ and that $W_{d,l-1}$ is contained in the image of $\Theta_d$. If $\Psi\in W_{d,l}$, then $t\mapsto \Psi(\cdot+t)$ is a polynomial of degree at most $l$ in $t\in\R$, so we find $Z_j:\Conv(\R^n,\R)\rightarrow\M(\R^n)$, $0\le j\le l$, such that
				\begin{align*}
					\Psi(f+t)=\sum_{j=0}^l t^jZ_j(f)
				\end{align*}
				for all $f\in\Conv(\R^n,\R)$. Using the inverse of the Vandermonde matrix, it is easy to see that $Z_j\in \P_{d-j}\LV(\R^n)^{tr}$ and that $t\mapsto Z_j(\cdot+t)$ is a polynomial of degree at most $l-j$. In particular, $Z_l$ is invariant under the addition of constants, i.e. $Z_l\in W_{d-l,0}$. If $Q$ corresponds to $Z_l$ under the isomorphism $W_{d,0}\cong \Poly_{d-l}((\R^n)^*)\otimes\MAVal_k(\R^n)$ from \autoref{lemma:verticaltranslationInvariant}, then $P \in \Poly_d(\A(n,\R))\otimes\MAVal(\R^n)$ defined by $P(s,v):=s^lQ(v)$ has the property that the highest order term of $t\mapsto [\Theta_d(P)](\cdot+t)$ is given by $Z_l$. Thus $\Psi-\Theta_d(P)\in W_{d,l-1}$, which shows that $\Psi$ is contained in the image of $\Theta_d$.
			\end{proof}
			Note that the result boils down to the following statement.
			\begin{corollary}
				\label{corollary:translationInvGradedAlg}
				The map $\Theta_d$ induces an isomorphism
				\begin{align*}
					\P_d\LV(\R^n)^{tr}/\P_{d-1}\LV(\R^n)^{tr}\cong \Sym^{d}(\A(n,\R)^*)_\C\otimes \MAVal(\R^n)
				\end{align*} of $\GL(n,\R)$-modules.
			\end{corollary}

		\subsection{Rigid motion invariant polynomial local functionals}	
		\label{section:RigidMotionInvariant}
		In order to obtain \autoref{maintheorem:rigidMotionInvariant} from the description of $\P_d\LV(\R^n)^{tr}$, we will decompose $\MAVal_k(\R^n)$ into irreducible subrepresentations of $\SO(n)$. This approach requires that this space is a continuous representation of $\SO(n)$ (with respect to one of our topologies), which follows directly from \autoref{proposition:contRepCompactToCompact} or \autoref{theorem:ClassificationTranslationInvariant}. We briefly recall the necessary representation theoretic notions but refer to \cite{GoodmanWallachRepresentationsinvariantsclassical2000} for the general background.\\
		Since $\SO(n)$ is compact, its irreducible representations are finite dimensional. For $n=2$, the irreducible representations of $\SO(2)$ are indexed by their weight $m\in \mathbb{Z}$. For $n\ge 3$, the irreducible representations of $\SO(n)$ are indexed by their highest weight, that is, by tuples $\lambda=(\lambda_1,\dots,\lambda_{\lfloor\frac{n}{2}\rfloor})\in \mathbb{Z}^{\lfloor\frac{n}{2}\rfloor}$ such that
		\begin{align}
			\label{eq:RangeHighestWeights}
			\begin{cases}
				\lambda_1\ge \lambda_2\ge\dots\ge \lambda_{\lfloor\frac{n}{2}\rfloor} & \text{for odd}~n,\\
				\lambda_1\ge \lambda_2\ge\dots\ge |\lambda_{\lfloor\frac{n}{2}\rfloor}| & \text{for even}~n.
			\end{cases}
		\end{align}
		We use $\Gamma_\lambda$ to denote any isomorphic copy of an irreducible representation with highest weight $\lambda\in\mathbb{Z}^{\lfloor\frac{n}{2}\rfloor}$ for $n\ge 3$, or $\Gamma_m$ for the irreducible representations of $\SO(2)$ with weight $m\in\mathbb{Z}$. The following is a simple adaptation of the results in \cite{AleskerEtAlHarmonicanalysistranslation2011}*{Section~3}. 	
		\begin{theorem}
			\label{theorem:MAValIsotypicDecomposition}
			Let $n\ge 3$ and  $0\le k\le n$. The space $\MAVal_k(\R^n)$ is a direct sum of the irreducible representations of $\SO(n)$ with highest weight $(\lambda_1,\dots,\lambda_{\lfloor\frac{n}{2}\rfloor})$ that precisely satisfy the conditions
			\begin{enumerate}
				\item $\lambda_j=0$ for $j>\max(k,n-k)$,
				\item $|\lambda_j|\ne 1$ for $1\le j\le \lfloor\frac{n}{2}\rfloor$,
				\item $\lambda_1\le 2$.
			\end{enumerate}
			In particular, the space $\MAVal_k(\R^n)$ is multiplicity free under the action of $\SO(n)$.
		\end{theorem}
		\begin{proof}
			Recall that $P\Lambda^{n-k,k}\subset \Lambda^{n-k}(\R^n)^*\otimes \Lambda^k((\R^n)^*)^*$ denotes the subspace of primitive differential forms, i.e. all $\tau\in \Lambda^{n-k}(\R^n)^*\otimes \Lambda^k((\R^n)^*)^*$ such that $\omega_s\wedge\tau=0$ for the natural symplectic form $\omega_s$ on $\R^n\times(\R^n)^*$. We let $L:\Lambda^*(\R^n\times(\R^n)^*)^*\rightarrow \Lambda^*(\R^n\times(\R^n)^*)^*$ denote the Lefschetz operator, i.e. multiplication with the symplectic form.\\
			By \autoref{theorem:IsmorphismPrimitiveFormsMAVal}, $P\Lambda^{n-k,k}\cong \MAVal_k(\R^n)$ as an $\SO(n)$-module. Moreover, the Lefschetz decomposition (see \cite{HuybrechtsComplexgeometry2005}*{Proposition~1.2.30}) implies that
			\begin{align*}
				&P\Lambda^{n-k,k}\oplus\Im \left(L:\Lambda^{n-k-1}(\R^n)^*\otimes \Lambda^{k-1}((\R^n)^*)^*\rightarrow\Lambda^{n-k}(\R^n)^*\otimes \Lambda^k((\R^n)^*)^*\right)\\
				&=\Lambda^{n-k}(\R^n)^*\otimes \Lambda^k((\R^n)^*)^*,
			\end{align*}
			where the second space is isomorphic to $\Lambda^{n-k-1}(\R^n)^*\otimes \Lambda^{k-1}((\R^n)^*)^*$ as an $\SO(n)$-module since $L$ is injective on this space. If we denote the character of a representation $V$ of $\SO(n)$ by $\mathrm{char}(V)$ and set $E_i=\mathrm{char}(\Lambda^i(\R^n))=\mathrm{char}(\Lambda^i(\R^n)^*)$, this implies
			\begin{align*}
				\mathrm{char}(\MAVal_k(\R^n))=E_{n-k}\cdot E_k-E_{n-k-1}E_{k-1}.
			\end{align*}
			It now follows from \cite[Corollary~3.4]{AleskerEtAlHarmonicanalysistranslation2011} that
			\begin{align*}
				\mathrm{char}(\MAVal_k(\R^n))=\sum_{\lambda}\mathrm{char}\overline{\Gamma}_\lambda,
			\end{align*}
			where 
			\begin{align*}
				\overline{\Gamma}_\lambda=\begin{cases}
					\Gamma_\lambda \oplus \Gamma_{(\lambda_1,\dots,\lambda_{\lfloor \frac{n}{2}\rfloor-1},-\lambda_{\lfloor \frac{n}{2}\rfloor})}, &n~\text{even},~\lambda_{\lfloor \frac{n}{2}\rfloor}\ne 0,\\
					\Gamma_\lambda, &\text{else},
				\end{cases}
			\end{align*}
			and the sum ranges over all $\lfloor n/2\rfloor$-tuples of non-negative integers satisfying Eq.~\eqref{eq:RangeHighestWeights} and
			\begin{enumerate}
				\item $\lambda_1\le 2$,
				\item $\lambda_i\ne 1$ for $1\le i\le \lfloor n/2\rfloor$,
				\item $\lambda_j=0$ for $j>\max(k,n-k)$.
			\end{enumerate}
			Since two representations are isomorphic if and only if their characters coincide, this shows the desired result using the definition of $\overline{\Gamma}_\lambda$.
		\end{proof}
	
		\begin{corollary}
			\begin{enumerate}
				\item $\MAVal_0(\R^2)$ is $\SO(2)$-irreducible with weight $0$.
				\item $\MAVal_2(\R^2)$ is $\SO(2)$-irreducible with weight $0$.
				\item $\MAVal_1(\R^2)$ is a multiplicity free representation of $\SO(2)$ with weights $-2,0,2$.
			\end{enumerate}			
		\end{corollary}
		\begin{proof}
			With the same reasoning as before, the Lefschetz decomposition provides an $\SO(2)$-equivariant isomorphism
			\begin{align*}
				\MAVal_k(\R^2)\oplus \Lambda^{1-k}(\R^2)^*\otimes \Lambda^{k-1}((\R^2)^*)^*\cong \Lambda^{2-k}(\R^2)^*\otimes \Lambda^{k}((\R^2)^*)^*.
			\end{align*}
			For $k=0$ or $k=2$, the second summand vanishes while the right hand side is isomorphic to the trivial representation, which shows the first two claims. For $k=1$, the second space is isomorphic to the trivial representation, while $\Lambda^1(\R^2)^*\cong \Lambda^1((\R^2)^*)^*\cong \C^2$ decomposes into two irreducible representations with weights $1$ and $-1$. Thus, in this case the right hand side is a sum of four irreducible representations: the weights $-2$ and $2$ occur with multiplicity $1$, whereas $0$ occurs with multiplicity $2$. This shows the claim.
			
		\end{proof}
	
		\begin{corollary}
			\label{corollary:dimensionInvariant}
			Let $d\ge 0$, $0\le k\le n$. Then for $n\ge 3$, or $n\ge 2$ and $k\ne 1$,
			\begin{align*}
				\dim (\Sym^d((\R^n)^*)_\C\otimes \MAVal_k(\R^n))^{\SO(n)}=\begin{cases}
					2 & d\ge 2~\text{even}, 1\le k\le n-1,\\
					1 & d\ge 2~\text{even}, k=0~\text{or}~k=n,\\
					1 & d=0,~0\le k\le n,\\
					0 & d~\text{odd}.
				\end{cases}
			\end{align*}
			For $n=2$, and $k=1$, 
				\begin{align*}
				\dim (\Sym^d((\R^2)^*)_\C\otimes \MAVal_1(\R^2))^{\SO(2)}=\begin{cases}
					3& d\ge 2~\text{even},\\
					1& d= 0~\text{even},\\
					0& d~\text{odd}.
				\end{cases}
			\end{align*}
		\end{corollary}
		\begin{proof}
			Assume first that $n\ge 3$.	It is well known that $\Sym^d((\R^n)^*)_\C\cong \Sym^d(\R^n)_\C\cong \bigoplus_{j=0}^{\lfloor\frac{d}{2}\rfloor}\Gamma_{(d-2j,0,\dots,0)}$, compare \cite{FultonHarrisRepresentationtheory1991}*{Section~19.5}. Since $\Sym^d((\R^n)^*)_\C$ is self dual as an $\SO(n)$-module,
			\begin{align*}
				&(\Sym^d((\R^n)^*)_\C\otimes \MAVal_k(\R^n))^{\SO(n)}
				\cong\mathrm{Hom}(\Sym^d((\R^n)^*)_\C,\MAVal_k(\R^n))^{\SO(n)}\\
				&\cong\sum_{j=0}^{\lfloor\frac{d}{2}\rfloor}\mathrm{Hom}(\Gamma_{(d-2j,0\dots,0)},\MAVal_k(\R^n))^{\SO(n)},
			\end{align*}
			and the result follows from \autoref{theorem:MAValIsotypicDecomposition} and Schur's lemma.\\
			For $n= 2$, we have  $\Sym^d((\R^2)^*)_\C\cong \Sym^d(\R^2)_\C\cong \bigoplus_{j=0}^{d}\Gamma_{d-2j}$ and the same argument as before shows the desired result.
		\end{proof}
		
		In the following lemma, we use again the identification $ \R\times(\R^n)^*\cong \A(n,\R)$ given by $(s,v)\mapsto \langle v,\cdot\rangle+s$. Recall also that $\Hess_k=\mathcal{E}^{0,0}_k$, $0\le k\le n$, denotes the $k$th Hessian measure (compare Eq.~\eqref{eq:definitionEabc}), which is an $\SO(n)$-invariant element of $\MAVal_k(\R^n)$.
		\begin{lemma}
			\label{lemma:rigidMotionLinearIndep}
			For $0\le k\le n$, $i\ge 0$ and $1\le l\le n-1$, $j\ge 1$, the elements
			\begin{align*}
				E^{a,i}_k(s,v)=&s^a|v|^{2i}\Hess_k\\
				F^{b,j}_l(s,v)=&s^b|v|^{2(j-1)}\frac{1}{l+1}\frac{d}{dt}\Big|_0\Hess_{l+1}(\cdot+t\langle v,\cdot\rangle^2)
			\end{align*}
			are linearly independent in $\Poly(\A(n,\R))\otimes \MAVal(\R^n)$. 
		\end{lemma}
		\begin{proof}
			By comparing the degrees of polynomiality in $s$ and $v$ as well as the degrees of homogeneity, it is sufficient to show that $|v|^2\Hess_k$ and $\frac{d}{dt}\big|_0\Hess_{k+1}(\cdot+t\langle v,\cdot\rangle^2)$ are linearly independent for every $1\le k\le n-1$, which is easily checked since the first defines an $\SO(n)$-invariant element of $\MAVal(\R^n)$ for every fixed $v\in \R^n$, which is not true for the second.
		\end{proof}
	
	\begin{proof}[Proof of \autoref{maintheorem:rigidMotionInvariant}]
		By \autoref{theorem:ClassificationTranslationInvariant} and \autoref{lemma:rigidMotionLinearIndep}, the functionals $\mathcal{E}^{a,b}_c$ and $\mathcal{F}^{a,b}_c$ are linearly independent for the given parameters since they are the images of $E^{a,b}_c$ and $F^{a,b}_c$ under the map $\Theta_d$. It is thus sufficient to show that they span $\P_d\LV(\R^n)$. This follows easily from \autoref{corollary:translationInvGradedAlg} and \autoref{corollary:dimensionInvariant}, since we have isomorphisms of $\SO(n)$-modules 
		\begin{align*}
			\P_d\LV(\R^n)^{tr}/\P_{d-1}\LV(\R^n)^{tr}&\cong \Sym^{d}(\A(n,\R)^*)_\C\otimes \MAVal(\R^n).
		\end{align*}
	\end{proof}
	Let us briefly comment on the case $n=2$. In this case, fix an orientation of $\R^2$ and let $I:\R^2\rightarrow \R^2$ denote the counterclockwise rotation by $\frac{\pi}{2}$. Then
	\begin{align*}
		\tilde{F}^{b,j}_1:=s^b|v|^{2(j-1)}\frac{1}{2}\frac{d}{dt}\Big|_0\Hess_{2}(\cdot+t\langle Iv,\cdot\rangle^2)
	\end{align*}
	provides the elements in $\Poly_d(\A(2,\R))\otimes \MAVal(\R^2)$ corresponding to the additional functionals required by \autoref{corollary:dimensionInvariant}.

\bibliographystyle{abbrv}
\bibliography{../../../library/library.bib}

\Addresses
\end{document}